\documentclass{gtpart}     
\usepackage{amscd}
\usepackage{epic,eepic, latexsym, xypic, url, color, epsfig}
\title{An Abel map to the compactified Picard scheme realizes Poincar\'e duality}

\author{Jesse Leo Kass}
\givenname{Jesse Leo}
\surname{Kass}
\address{Current: J.~L.~Kass, Dept.~of Mathematics, University of South Carolina, 1523 Greene Street, Columbia, SC 29208}
\address{Former: J.~L.~Kass, Leibniz Universit\"{a}t Hannover, Institut f\"{u}r algebraische Geometrie, Welfengarten 1, 30060 Hannover, Germany}
\email{kassj@math.sc.edu}
\urladdr{http://www2.iag.uni-hannover.de/~kass/}

\author{Kirsten Wickelgren}
\givenname{Kirsten}
\surname{Wickelgren}
\address{Current: K.~Wickelgren, School of Mathematics, Georgia Institute of Technology, 686 Cherry Street, Atlanta, GA 30332}
\address{Former: K.~Wickelgren, Dept.~of Mathematics, Harvard University, One Oxford Street, Cambridge, MA 02138}
\email{wickelgren@post.harvard.edu}
\urladdr{http://people.math.gatech.edu/~kwickelgren3/}

\volumenumber{}
\issuenumber{}
\publicationyear{}
\papernumber{}
\startpage{}
\endpage{}
\doi{}
\MR{}
\Zbl{}
\received{}
\revised{}
\accepted{}
\published{}
\publishedonline{}
\proposed{}
\seconded{}
\corresponding{}
\editor{}
\version{}

\newtheorem{tm}[subsection]{Theorem}
\newtheorem{pr}[subsection]{Proposition}
\newtheorem{lm}[subsection]{Lemma}
\newtheorem{cor}[subsection]{Corollary}
\newtheorem{df}[subsection]{Definition}
\newtheorem{rmk}[subsection]{Remark}
\newtheorem{hyp}[subsection]{Hypothesis}
\newtheorem{exa}[subsection]{Example}
\newtheorem{claim}[equation]{Claim}

\renewcommand{\mathbf}{\mathbb}

\newcommand{\ab}{{\rm ab}}
\newcommand{\rH}{{\rm H}}

\renewcommand{\Z}{\mathbb{Z}}
\newcommand{\Zhat}{\hat{\mathbb{Z}}}

\newcommand{\kbar}{\overline{k}}

\newcommand{\G}{\mathbb{G}}

\newcommand{\proj}{\mathbb P}

\newcommand{\Gal}{\operatorname{Gal}}

\newcommand{\Tr}{\operatorname{Tr}}

\newcommand{\Hom}{\operatorname{Hom}}

\newcommand{\Pic}{\operatorname{Pic}}
\newcommand{\NS}{\operatorname{NS}}

\newcommand{\Spec}{\operatorname{Spec}}

\newcommand{\Pres}{\operatorname{Pres}}

\newcommand{\Mor}{{\text{Mor}}}

\renewcommand{\cl}{\operatorname{cl}}

\newcommand{\Ker}{\operatorname{Ker}}
\newcommand{\red}{\operatorname{red}}
\newcommand{\Coker}{\operatorname{Coker}}

\newcommand{\Jacbar}{\operatorname{\overline{P}ic}}
\newcommand{\Jac}{\operatorname{Pic}}
\newcommand{\bound}{\partial X}

\newcommand{\bbP}{\mathbf{P}}

\newcommand{\bbZ}{\mathbf{Z}}

\newcommand{\calO}{{ \mathcal O}}

\newcommand{\calV}{{ \mathcal V}}
\newcommand{\calW}{{ \mathcal W}}

\newcommand{\PtoPic}{p_1}
\newcommand{\PtoJbar}{p_2}
\newcommand{\Xsing}{X^{+}}
\newcommand{\Abel}{\operatorname{Ab}}
\newcommand{\AbelTilde}{\operatorname{\widetilde{A}b}}
\newcommand{\AbelLift}{ \operatorname{Pres}(\Abel)}
\newcommand{\can}{\sigma}

\newcommand{\hidden}[1]{\footnote{Hidden:  #1}}
\renewcommand{\hidden}[1]{}

\begin{document}

\begin{abstract}    

For a smooth algebraic curve $X$ over a field, applying $\rH_1$ to the Abel map $X \to \Jac X/\partial X$ to the Picard scheme of $X$ modulo its boundary realizes the Poincar\'e duality isomorphism $$\rH_1(X, \Z/ \ell) \to  \rH^1(X/ \partial X, \Z/\ell(1) )  \cong \rH^1_c(X, \Z/\ell(1)).$$ We show the analogous statement for the Abel map $X/\partial X \to \Jacbar X/\partial X$ to the compactified Picard, or Jacobian, scheme, namely this map realizes the Poincar\'e duality isomorphism $\rH_1(X/ \partial X, \Z/\ell) \to \rH^1(X, \Z/\ell(1))$. In particular, $\rH_1$ of this Abel map is an isomorphism. 

In proving this result, we prove some results about $\Jacbar$ that are of independent interest.  The singular curve $X/\partial X$ has a unique singularity that is an ordinary fold point, and we describe the compactified Picard scheme of such a curve up to universal homeomorphism using a presentation scheme. We construct a Mayer--Vietoris sequence for certain pushouts of schemes, and an isomorphism of functors $\pi_1^{\ell} \Pic^0(-) \cong H^1(-,\Z_{\ell}(1)).$ 

\end{abstract}

\maketitle


\section{Introduction}
In this paper we extend a classical relation between Poincar\'{e} duality and the Abel map.  The classical Abel map  of a smooth proper curve over a field $k$ is a map, $\Abel$, from the curve $X$ to its Picard, or Jacobian, scheme $\Jac X$.  The Picard scheme $\Jac X$ is the moduli space of invertible sheaves, and $\Abel$ sends a point $x$ to the sheaf  $\calO_{X}(x)$ of rational functions with at worst a pole at $x$.  The Abel map induces a homomorphism from the homology of $X$ to the homology of $\Jac X$. Moreover, there is a canonical isomorphism $\can$ betweeen the homology of $\Jac X$ and the cohomology of $X$, making the composition  $\can^{-1} \circ H_1(\Abel) $ a homomorphism between the homology and the cohomology of $X$.  This composition is exactly the Poincar\'{e} duality isomorphism
\begin{equation} \label{Eqn: PDForProper}
	\text{Poincar\'{e} Duality} = \can^{-1} \circ \rH_{1}(\text{Abel Map}).
\end{equation}
In other words, the Abel map of a smooth proper curve realizes the Poincar\'{e} duality isomorphism. In this paper we extend this result by showing that the Poincar\'{e} duality isomorphism associated to a smooth nonproper curve is realized by the Abel map from an explicit singular proper curve $\Xsing$ to its compactified Picard scheme. The singular curve $\Xsing$ plays the role of the curve modulo its boundary, as in topological Poincar\'e duality of manifolds. 

When $X$ is smooth and proper, the isomorphism $\can$ is constructed as follows.  Fix a prime $\ell$ distinct from the characteristic $ \operatorname{char} k$.  The \'{e}tale cohomology of $X$ with coefficients in the Tate twist $\bbZ/\ell(1) = \mu_{\ell}$, or $\ell$th roots of unity, is recorded by $\Jac X$ in the following manner.  The points of $\Jac X$ are invertible sheaves  or, equivalently, $\G_m$-torsors and thus correspond to elements of $\rH^1(X, \G_m)$. The inclusion $\bbZ/\ell(1) \to \G_{m}$ induces an isomorphism 
\begin{equation} \label{Eqn: TorsionToCoh}
	 \rH^1(X_{\kbar}, \Z/\ell(1)) \stackrel{\cong}{\to} \Jac X [\ell] (\kbar)
\end{equation}
from \'{e}tale cohomology to the $\ell$-torsion subgroup of the group of $\kbar$-valued points of $\Jac X$, as is seen with the long exact sequence associated to the Kummer exact sequence 
$$
	\xymatrix{1 \ar[r] & \Z/\ell(1) \ar[r] & \G_m \ar[rr]^{z \mapsto z^{\ell}} && \G_m \ar[r] & 1}.
$$ 
Thus, the moduli definition of the Picard scheme identifies its torsion points with $\rH^1(X_{\kbar}, \Z/\ell (1))$.

The structure of the Picard scheme computes these torsion points in terms of its fundamental group. Let $\Jac^0 X$ denote the connected component of $\Jac X$ containing the identity $e$, and note that the open immersion $\Jac^0 X \to \Jac X$ determines a canonical isomorphism $\pi_{1}(\Jac^0 X_{\kbar},e) \to \pi_1(\Jac X_{\kbar},e)$ between the fundamental group of $\Jac^0 X_{\kbar}$ based at the identity and the fundamental group of $\Jac X_{\kbar}$ based at the identity. The  multiplication-by-$\ell$ map $\Jac^0 X \to \Jac^0 X$, defined $I \mapsto I^{\otimes \ell}$, is a Galois covering space map with Galois group equal to the group of $\ell$-torsion, inducing a homomorphism 
\begin{equation} \label{Eqn: PiToTorsion}
	\pi_{1}( \Jac X_{\kbar},e) \to \Jac X [\ell] (\kbar)
\end{equation}
from the fundamental group of $\Jac X_{\kbar}$ to $\Jac X [\ell] (\kbar)$.  The Serre--Lang Theorem implies that \eqref{Eqn: PiToTorsion} becomes an isomorphism after tensoring $\pi_{1}( \Jac X_{\kbar},e)$ with $\Z/\ell$.  Since the fundamental group is abelian, for instance by Serre--Lang or alternatively by the Eckmann--Hilton argument, for any definition of homology satisfying the Hurewicz theorem, there is an isomorphism
\begin{equation} \label{Eqn: H_1ToTorsion}
	H_{1}(\Jac X_{\kbar}, \bbZ/\ell) \stackrel{\cong}{\to} \Jac X [\ell] (\kbar).
\end{equation}
Define $\can$ to be the composition of \eqref{Eqn: TorsionToCoh} with the inverse of \eqref{Eqn: H_1ToTorsion}. Equation~\eqref{Eqn: PDForProper} is then a consequence of \cite[Dualit\'e, Proposition~3.4]{sga4andhalf}.  

In other words, the canonical isomorphism $\sigma$ results from combining the moduli definition of $\Pic$ with its structure theory.

In this paper we extend the relation between the classical Abel map and Poincar\'e duality for a smooth proper curve to a relation between the Altman--Kleiman Abel map and Poincar\'e duality for a smooth nonproper curve, which we place in analogy with manifolds with boundary. Topological Poincar\'e duality works for manifolds with boundary, resulting in isomorphisms between the homology of the manifold modulo its boundary and the cohomology of the manifold. We will discuss the algebraic analogue below, in which to a smooth nonproper algebraic curve $X$ (satisfying a technical assumption), we associate a singular curve $\Xsing$ which should be thought of as $X$ modulo a sort of natural boundary. 

The Altman--Kleiman Abel map does not embed $\Xsing$ into its Picard scheme, but rather embeds $\Xsing$ into its compactified Picard scheme. Like the Picard scheme, the compactified Picard scheme has a moduli definition and structure theory, of course, but the structure of the compactified Picard scheme is more complicated than that of the Picard scheme, and is studied in this paper.  



\subsection{The main result} This paper computes the structure of the compactified Picard scheme of $\Xsing$, and uses this computation to show that Poincar\'e duality for a nonproper smooth curve $X$ is realized by the Abel map of $\Xsing$, extending the result discussed above. Embedded in this statement is the claim that the compactified Picard scheme admits a canonical isomorphism $\sigma: \rH^1(X_{\kbar}, \Z/\ell(1)) \stackrel{\cong}{\to}  \rH_{1}(\Jacbar \Xsing_{\kbar}, \Z/\ell)$. We develop a structure theory for $\Jacbar \Xsing$ which when combined with the moduli definition of $\Pic$ allows us to define $\sigma$.




It is natural to consider a smooth curve modulo its boundary in the context of Poincar\'e duality. The Poincar\'e duality perfect pairing for an oriented manifold $M$  with boundary $\partial M$ $$\rH^i (M/ \partial M, \Z/\ell) \otimes \rH^{\operatorname{dim} M - i} (M, \Z/\ell) \to \Z/\ell$$ follows from Michael Atiyah's formula for the dual of $M/ \partial M$ in terms of the tangent bundle of $M$ in the category of spectra \cite{Atiyah}. In this sense, $M/\partial M$ is dual to a shift of $M$ itself. The analogous duality in algebraic geometry, or rather $\proj^1$-spectra \cite{Hu}, produces perfect pairings of \'etale cohomology groups \cite{Isaksen}.  In particular, let $X$ be a smooth curve and assume that $X$ is an open subscheme of a smooth proper curve $\tilde{X}$ such that the residue fields of the points of $\tilde{X} - X$ are separable extensions of $k$. (For instance, when $k$ has characteristic $0$, this assumption is automatically satisfied.) Form $X/ \partial X := \Xsing$, where $\Xsing$ is defined by the pushout diagram \cite[Theorem~5.4]{ferrand03} $$\xymatrix{ \partial X \ar[r] \ar[d] &  \tilde{X} \ar[d] \\ \Spec k \ar[r] & \Xsing},$$ and $\partial X$ denotes the complement $\tilde{X} -X$ with its reduced closed subscheme structure. (This pushout is discussed further in Section~\ref{Section: FoldSingularities}.) Let $R$ denote $\Z/\ell$ or $\Z_{\ell}$, with $\ell$ relatively prime to $\operatorname{char} k$. By \cite[VI, Section~11, Corollary~11.2]{Milnebook} and the canonical isomorphism $\rH^1_c(X_{\kbar}, R(r)) \cong \rH^1( \Xsing_{\kbar}, R(r))$\hidden{

\begin{pr}
	There is a canonical isomorphism $\rH^1_c(X_{\kbar}, R(r)) \cong \rH^1(X/(\partial X)_{\kbar}, R(r))$.
\end{pr}

\begin{proof}
We may assume that $k = \kbar$, because the formation of $X/(\partial X)$ commutes with arbitrary field extension \cite[Lemma~4.4]{ferrand03}. Let $\Xsing = X/(\partial X)$, and $j: X \to \Xsing$ be the open immersion. By \cite[III, Section~1, p.~93]{Milnebook}, there is a canonical isomorphism $\rH^1_c(X, R(r)) \cong \rH^1(\Xsing, j_! R(r))$.  Let $i: \Spec k \to \Xsing$ be a closed immersion whose set-theoretic image is the singular point. Note that $i^* R(r) = R(r)$ and $j^* R(r)$, giving an exact sequence $$0 \to j_! R(r) \to R(r) \to i_* R(r) \to 0$$ by \cite[II, Section~3, p.~76]{Milnebook}. Applying $\rH^*(\Xsing, -)$ to this exact sequence yields $$\to \rH^0(\Xsing, R(r) ) \to \rH^0(\Xsing, i_*R(r) ) \to \rH^1_c(X, R(r))  \to \rH^1(\Xsing, R(r) ) \to \rH^1(\Xsing, i_* R(r)) \to .$$   Since $\rH^0(\Xsing, i_*R(r) ) \cong \rH^0(\Spec k, R(r)) \cong R(r)$, the map $\rH^0(\Xsing, R(r) ) \to \rH^0(\Xsing, i_*R(r) )$ is surjective, whence $ \rH^1_c(X, R(r))  \to \rH^1(\Xsing, R(r) ) $ is injective. Since $$\rH^1(\Xsing, i_*R(r) ) \cong \rH^1(\Spec k, R(r)) \cong 0,$$ we have that $ \rH^1_c(X, R(r))  \to \rH^1(\Xsing, R(r) ) $ is surjective.
\end{proof}

}, there is a perfect pairing of \'etale cohomology groups $$\rH^1( \Xsing_{\kbar}, R(r) ) \otimes \rH^1(X_{\kbar}, R(1-r)) \to R.$$ 

For simplicity, define $\rH_1( (-)_{\kbar} , R)$ for geometrically connected $k$-schemes using the abelianization of the $\ell$-\'etale fundamental group by $\rH_1( (-)_{\kbar} , R) := \pi_1^{\ell}(-)_{\kbar}^{\ab} \otimes_{\Z_{\ell}} R$, but see \cite[Section~7]{Friedlander} for a discussion of \'etale homology in more generality. The tautological pairing between $\rH_1$ and $\rH^1$\hidden{in terms of the \'etale fundamental group and torsors, this is the tautological action on an automorphism of the fiber functor on the fiber} produces the Poincar\'e duality isomorphisms (see Section~\ref{sectionAJPD}):

\begin{equation} \label{PoincarŽforXtoPicX+} 
	\text{Poincar\'{e} Duality}:  \rH_1(X_{\kbar}, R) \to \rH^1( \Xsing_{\kbar}, R(1)).
\end{equation}
\begin{equation}\label{PoincarŽforX+toJbar} 
	\text{Poincar\'{e} Duality}:  \rH_1( \Xsing_{\kbar}, R) \to \rH^1(X_{\kbar}, R(1))
\end{equation} 

It is not hard to modify the argument given in the proper case to show that an Abel map realizes the isomorphism \eqref{PoincarŽforXtoPicX+}.  More explicitly, the moduli space $\Pic \Xsing$ of invertible sheaves on $\Xsing$ exists, and the rule $x \mapsto \calO_{\Xsing}(x)$ defines an Abel map $$\Abel \colon X \to \Pic \Xsing$$ whose domain is the nonproper curve $X$.  The connected component of $\Pic \Xsing$ containing the trivial line bundle is a semi-abelian variety, and a generalization of the Serre--Lang Theorem \cite [Appendix, p.~66~(III)]{Mochizuki} computes the fundamental group of $\Pic \Xsing$ in terms of torsion points.  Using the fact that these torsion points are line bundles, a canonical isomorphism $\can$ can be constructed as before, and this isomorphism satisfies \eqref{Eqn: PDForProper}.

What about the isomorphism \eqref{PoincarŽforX+toJbar}?  The rule $x \mapsto \calO_{X}(x)$ does not define an Abel map from $\Xsing$ to $ \Pic \Xsing$ because $\calO_{X}(x)$ is undefined when $x$ is a singularity.  The generalized Picard scheme $\Pic \Xsing$ is not proper, but it is contained in the proper scheme $\operatorname{\overline{M}od} \Xsing$ defined as the moduli space of rank $1$, torsion-free sheaves on $\Xsing$.  The ideal sheaf $I_{x} = \calO_{\Xsing}(-x)$ of a point on $\Xsing$ is a rank $1$, torsion-free sheaf, and so the rule $x \mapsto I_{x}$ defines a map $\Xsing \to \operatorname{\overline{M}od} \Xsing$.  The scheme $\operatorname{\overline{M}od} \Xsing$ has the undesirable property that the connected components are reducible, so rather than work with $\operatorname{\overline{M}od} \Xsing$, we work with the Zariski closure of the line bundle locus which we call the compactified Picard scheme $\Jacbar \Xsing$.  The rule $x \mapsto I_{x}$ defines a morphism $\Abel \colon \Xsing \to \Jacbar \Xsing$, that will be called the Abel map of $\Xsing$.

The main theorem of this paper is that $\Abel$ realizes the Poincar\'{e} duality isomorphism \eqref{PoincarŽforX+toJbar}.  The theorem applies under the following hypothesis:
\begin{hyp} \label{Hypothesis}
	The curve $X$ has the property that
	\begin{enumerate}
		\item the normal proper model $\widetilde{X}$ of $X$ is smooth over $k$;
		\item the extension of residue fields $k(x)/k$ is separable for every point $$x \in \partial X := \widetilde{X} - X.$$
	\end{enumerate}
\end{hyp}

The theorem states:

\begin{tm}\label{Thm: MainIntro}
If $X$ is a smooth curve over $k$ that satisfies Hypothesis~\ref{Hypothesis}, then 
	\begin{equation} \label{Eqn: PDForNonProper}
		-\text{Poincar\'{e} Duality} = 	\can^{-1} \circ \rH_{1}(\text{Abel Map})
	\end{equation}
	for an explicit isomorphism of $\Gal(\kbar/k)$-modules $\can: \rH^{1}(X_{\kbar}, R(1)) \to \rH_{1}(\Jacbar \Xsing_{\kbar}, R)$.
\end{tm}
This result is Theorem~\ref{minus_wp_is_sigma_inverse_H1_Abel} below. Observe that in Equation~\eqref{Eqn: PDForNonProper}, Poincar\'{e} duality appears with a minus sign.  We discuss the significance of this sign and Hypothesis~\eqref{Hypothesis} later in this introduction.

We prove Theorem~\ref{Thm: MainIntro} using a strategy similar to the one used to prove Equation~\eqref{Eqn: PDForProper}. The isomorphism $\can$ is constructed using a description of the structure of $\Jacbar \Xsing$, and the definition of $\Pic \widetilde{X}$ as a moduli space. We then prove Equation~\eqref{Eqn: PDForNonProper} using analogues of results of \cite[Dualit\'e, Section~3]{sga4andhalf}. We also obtain a generalization of \cite[Dualit\'e, Propositions~3.2 and 3.4]{sga4andhalf} in the case of a proper smooth curve $\widetilde{X}$. Together, Propositions 3.2 and 3.4 imply that the Poincar\'e dual of a loop on $\widetilde{X}$ can be described as the associated monodromy on the pullback of the multiplication-by-$\ell$ map $\Pic \to \Pic$ under a translation of the Abel map. We give the monodromy of a path in terms of Poincar\'e duality on a quotient curve. The different fibers have a canonical identification after tensoring with the cohomology of the quotient curve, allowing us to speak of monodromy as an element of this cohomology. See Section~\ref{sectionAJPD}.

\subsection{New results about $\Jacbar \Xsing$}
In proving Theorem~\ref{Thm: MainIntro}, we prove new results about the compactified Picard scheme that are of independent interest. Unlike the Picard scheme of a smooth proper curve, the compactified Picard scheme does not have a group structure, so we cannot use e.g.,~the Serre--Lang Theorem, and instead compute the structure of $\Jacbar \Xsing$ up to universal homeomorphism.  Prior to this paper, very little was known about $\Jacbar \Xsing$ when $\partial X$ contains at least three points.  In this case, the singularity of $\Xsing$ is non-Gorenstein and non-planar.  While there is a large body of work describing the structure of the compactified Picard scheme of a curve with planar singularities, the only results that apply to $\Xsing$ that the authors are aware of are Altman--Kleiman's result that $\Jacbar \Xsing$ is a projective scheme \cite{altman80} and Kleiman--Kleppe's result that the moduli space $\operatorname{\overline{M}od}^{d} \Xsing$ of degree $d$ rank $1$, torsion-free sheaves is reducible \cite{Kleppe81}.  

To construct $\can$ in Theorem~\ref{Thm: MainIntro}, we describe the structure of $\Jacbar \Xsing$ as follows. Let $f: \widetilde{X} \to \Xsing$ denote the map from the pushout definition of $\Xsing$, which is also the normalization map. Let $x_0$ denote the singular point of $\Xsing$. There is a projective bundle $\Pres \Xsing$ over the Picard scheme $\Jac^{-1} \widetilde{X}$ representing a certain presentation functor, and a birational morphism $\Pres \Xsing \to \Jacbar^{-1} \Xsing$ from the bundle to the compactified Picard scheme.  The bundle map $\Pres \Xsing \to \Jac^{-1} \widetilde{X}$ admits sections labeled by the points of $f^{-1}(x_0)$, and the birational morphism $\Pres \Xsing \to \Jacbar^{-1} \Xsing$ is the map that, up to universal homeomorphism, identifies these sections in the sense that a natural pushout diagram, Diagram~\eqref{defineJnat}, defines a universal homeomorphism.  This is Theorem~\ref{Proposition: UnivesalHomeo}.  

The geometric description of $\Jacbar^{-1} \Xsing$ given by Theorem~\ref{Proposition: UnivesalHomeo} is of independent interest.  For example, the theorem (or more precisely Proposition~\ref{Prop: ClassificationOfJacbar} which is used to prove the theorem) shows that a rank $1$, torsion-free sheaf  lies in $\Jacbar^{-1} \Xsing$, rather than in a different component of the reducible scheme $\operatorname{\overline{M}od}^{-1} \Xsing$, precisely when the sheaf is the direct image of a line bundle under a partial normalization map $\overline{Y} \to \Xsing$ out of a curve $\overline{Y}$ with at most one singularity.  In particular, we have:
\begin{cor} 
	Assume $\operatorname{char}(k)>3$.  Define $\Xsing$ to be the rational curve with a $4$-fold point that is obtained from $\widetilde{X} := \bbP^{1}$ by identifying the points $\partial X = \{ \pm1, \pm2 \}$.  Define $\overline{Y}$ to be the rational curve with two nodes that is obtained from $\widetilde{X}$ by identifying $1$ with $-1$ and $2$ with $-2$.  If $g \colon \overline{Y} \to \Xsing$ is the natural morphism, then
	$$
		I = g_{*} \calO_{\overline{Y}}
	$$
	is not the limit of line bundles.
\end{cor}

Note that when $\partial X$ consists of two points, the singularity of $\Xsing$ is a node. The projective bundle $\Pres \Xsing$ that appears in Theorem~\ref{Proposition: UnivesalHomeo} is a generalization of the presentation scheme of a nodal curve constructed by Oda--Seshadri \cite{oda} and Altman--Kleiman \cite{altman90}, and as such, is also called the presentation scheme.  The presentation scheme of a nodal curve was also studied by Bhosle in \cite{bhosle}, where the scheme appears as an example of a moduli space of generalized parabolic bundles.    

Let us use what has been said about the structure of $\Jacbar^{-1} \Xsing$ to define $\can$. Assume for simplicity that $k = \kbar$, and that $\widetilde{X}$ has genus greater than $0$. As above, we have the presentation scheme $\Pres \Xsing$ and maps $$\xymatrix{ & \ar[dl]_{\PtoPic} \Pres \Xsing \ar[dr]^{\PtoJbar}&  \\  \Jac^{-1} \widetilde{X} && \Jacbar^{-1} \Xsing,}$$ with $\PtoPic$ a projective bundle, equipped with sections in bijection with the points of $\partial X$, or equivalently in bijection with the points of $f^{-1}(x_0)$, and $\PtoJbar$ a quotient map identifying the images of the sections in a certain manner. As $\widetilde{x}$ varies over all points of $\partial X$, apply the corresponding section to $\Abel ( \widetilde{x})$ to produce a set $\mathcal{E}$ of points of $\Pres \Xsing$. In the description of $\PtoJbar$, we will see that the points of $\mathcal{E}$ all have the same image under $\PtoJbar$, inducing a map from the \'etale fundamental groupoid $\pi_1 ( \Pres \Xsing, \mathcal{E})$ to the \'etale fundamental group of $\Jacbar^{-1} \Xsing$. Composing with the Hurewicz map, which is tautological with our definition of $\rH_1$, yields 
\begin{equation} \label{Eqn: Presgroupoid_toH1}
	\pi_1^{\ell} ( \Pres \Xsing_{\kbar}, \mathcal{E}) \to \rH_1(\Jacbar^{-1} \Xsing_{\kbar}, R).
\end{equation}

The projective bundle $\PtoPic$ identifies this fundamental groupoid with the fundamental groupoid of $\Pic^{-1} \widetilde{X}$ based at the images of the points of $\partial X$ under the Abel map $${\PtoPic}_*: \pi_1^{\ell} ( \Pres \Xsing, \mathcal{E}) \stackrel{\cong}{\to} \pi_1^{\ell}(\Pic^{-1} \widetilde{X}, \Abel_* (\partial X_{\kbar})).$$

In Theorem~\ref{Prop_groupoid_to_H1},  we show that the moduli definition of $\Pic$ yields a canonical isomorphism $$ \mathcal{F}_R\pi_1^{\ell}(\Pic^{-1} \tilde{X}, \Abel_* (\partial X_{\kbar})) \to \rH^1(X_{\kbar}, R(1))$$ from the free $R$-module $\mathcal{F}_R\pi_1^{\ell}(\Pic^{-1} \tilde{X}_{\kbar}, \Abel_* (\partial X_{\kbar}))$ on the fundamental groupoid $$\pi_1^{\ell}(\Pic^{-1} \tilde{X}_{\kbar}, \Abel_* (\partial X_{\kbar}))$$ to the \'etale cohomology group $\rH^1(X_{\kbar}, R(1))$. 

The map \eqref{Eqn: Presgroupoid_toH1} and universal property of $\mathcal{F}_R$ then give $$\sigma: \rH^1(X_{\kbar}, R(1)) \to \rH_1(\Jacbar \Xsing_{\kbar}, R),$$ which is shown to be an isomorphism in  Proposition~\ref{Prop_H1XR(1)toH1jbar_iso}. This is the $\can$ which appears in Theorem~\ref{Thm: MainIntro}.

\subsection{Connections with autoduality}
Theorem~\ref{Thm: MainIntro} is related to the theory of autoduality of the compactified Picard scheme.  Under various hypotheses on $\overline{X}$, Arinkin, Esteves--Gagn{\'e}--Kleiman, Esteves--Rocha, Melo--Rapagnetta--Viviani have proved that $\Pic (\Abel)$ is an isomorphism between the Picard schemes $\Pic^{0} \overline{X}$ and $\Pic^{0} \Jacbar \overline{X}$ \cite{Arinkin, esteves2002, Esteves_Flavio, MRV}.  How are those results related to the results of this paper?  

One consequence of Theorem~\ref{Thm: MainIntro} is
\begin{cor} \label{Cor: Intro}
	Applying either the functor $\rH_1((-)_{\kbar}, R)$ or $\rH^1((-)_{\kbar}, R)$ to $$\Abel: X^+ \to \Jacbar \Xsing$$ produces an isomorphism.
\end{cor}
For cohomology, the result is established using a tautological pairing between $$\rH_1((-)_{\kbar}, R)\text{ and }\rH^1((-)_{\kbar}, R),$$ discussed in the second paragraph of Section~\ref{sectionAJPD} and Appendix~\ref{appendix_MV} \eqref{H1YR=HomH_1R}.  The pairing is induced from the monodromy pairing between torsors and $\pi_1$. 

Using the following proposition, results similar to Corollary~\ref{Cor: Intro} can be deduced from autoduality results because we have:
\begin{pr} \label{natl_iso_pi_1Pic=H1}
	Let $k$ be a perfect field. There is a natural isomorphism of functors from proper, geometrically connected schemes over $k$ to $\Gal(\kbar/k)$-modules $$\pi_1^{\ell} (\Pic^0 (-)_{\kbar}, e) \cong \rH^1((-)_{\kbar}, \Z_{\ell}(1)).$$
\end{pr}
We provide a proof of Proposition~\ref{natl_iso_pi_1Pic=H1} in Appendix~\ref{appendixa}.

The above proposition implies that Corollary~\ref{Cor: Intro} remains valid when $\Xsing$ is replaced by a curve $\overline{X}$ whose compactified Picard scheme satisfies autoduality.  Currently, autoduality results have only been proven under the hypothesis that $\overline{X}$ has plane curve singularities, and $\Xsing$ has plane curve singularities exactly when $x_0 \in \Xsing$ is a node.  For such an $\Xsing$, \cite[Theorem~2.1]{esteves2002} implies that $\Abel$ induces an isomorphism on homology and cohomology, or alternatively see \cite[Theorem~C]{Arinkin_line_bundles}.   

\subsection{Concluding remarks}
Let us conclude our discussion with two remarks about Theorem~\ref{Thm: MainIntro}.  First, in the theorem, we assume $X$ satisfies Hypotheses~\ref{Hypothesis}. This assumption allows us to assert that the curve $\Xsing_{\kbar}$ obtained by extending scalars to $\kbar$ is obtained from a smooth curve by a suitable pushout, and in particular, allows us to avoid curves that are normal but not geometrically normal (i.e.,~not smooth).  Second, the conclusion of Theorem~\ref{Thm: MainIntro} differs from Equation~\eqref{Eqn: PDForProper} by the minus sign of  Equation~\eqref{Eqn: PDForNonProper}, because the Abel map of Equation~\eqref{Eqn: PDForNonProper} is not the classical Abel map appearing in Equation~\eqref{Eqn: PDForProper}.  Classically, the Abel map is defined by $x \mapsto \calO_{X}(x)$, while the Abel map in \eqref{Eqn: PDForNonProper} is defined by  $x \mapsto I_{x}=\calO_{X}(-x)$.  On a smooth proper curve, the difference is a matter of notation as they differ by the automorphism $I \mapsto \Hom(I, \mathcal{O}_{X})$.  On such a singular curve  $\Xsing$, however, the difference is significant as only the second rule necessarily defines a regular map, as we show in Example~\ref{ex_Abeldual_can_fail_to_extend}.

This paper is organized as follows. In Section~\ref{Section: FoldSingularities} we record the definition of the one point compactification $\Xsing$ of a suitable smooth curve $X$ over $k$.  We study the  compactified Picard scheme $\Jacbar \Xsing$ of $\Xsing$ in the next two sections.  In Section~\ref{Section: PresentationScheme} we define the presentation scheme  and then use it in Theorem~\ref{Proposition: UnivesalHomeo} to compute the compactified Picard scheme up to universal homeomorphism.  We study the Abel map in Section~\ref{Section_Abel_map}, where we  prove that the classical Abel map can fail to extend to a morphism out of $\Xsing$ and that the Altman--Kleiman Abel map lifts to a morphism into the presentation scheme.  

In the last three sections, we describe the homology of $\Jacbar^{-1} \Xsing$ and apply that description to prove Theorem~\ref{Thm: MainIntro}.  In Section~\ref{Section:CohXfromPic} we prove that the cohomology of the smooth curve $X$ can be recovered from the fundamental groupoid of $\Jac^{-1} \widetilde{X}$.  We use this fact in Section~\ref{section_H1Jbar} to construct the isomorphism $\sigma$ appearing in Equation~\eqref{Eqn: PDForNonProper}.  We prove that Equation~\eqref{Eqn: PDForNonProper} holds in Section~\ref{sectionAJPD}.  There are two appendices. Appendix~\ref{appendixa} proves Proposition~\ref{natl_iso_pi_1Pic=H1}, identifying cohomology and the fundamental group of the Picard scheme. Appendix~\ref{appendix_MV} proves a Mayer--Vietoris sequence for pushouts by closed immersions and finite maps. This sequence is of cohomology groups, or homology groups in small dimensions. The homology sequence is used in Sections~\ref{section_H1Jbar} and \ref{sectionAJPD}.

{\bf Acknowledgements:} We wish to thank Dennis Gaitsgory and Carl Mautner for useful discussions. We thank Karl Schwede for explaining seminormality to us, and Anton Geraschenko for helpful information about coproducts of schemes.  We also thank Alastair King for informing us of the work of Bhosle, and Nicola Pagani, Emily Riehl, and Nicola Tarasca for helpful feedback concerning exposition.

During the writing of this paper, the first author was a  Wissenschaftlicher Mitarbeiter at the Institut f\"{u}r Algebraische Geometrie, Leibniz Universit\"{a}t Hannover.  The first author was supported by an AMS-Simons Travel Grant, and the second author is supported by an American Institute of Mathematics $5$-year fellowship.

\section*{Conventions}
$k$ is a field.

$\overline{k}$ is a fixed algebraic closure of $k$.

A \textbf{curve} $X/k$ is a separated, finite type, geometrically integral $k$-scheme of pure dimension $1$.

If $T$ is a $k$-scheme, then a \textbf{family of rank $1$, torsion-free} on a curve $X$ parameterized by $T$ is a $\calO_{T}$-flat finitely presented $\calO_{X_T}$-module $I$ with the property that the restriction to any fiber of $X_{T} \to T$ is rank $1$ and torsion-free.

The degree $d$ \textbf{compactified Picard scheme} $\Jacbar^{d}(\overline{X})$ of a proper curve $\overline{X}$ is the Zariski closure of the line bundle locus in the moduli space of rank $1$, torsion-free sheaves of degree $d$. 

A $T$-\textbf{relative effective Cartier divisor}  is a $T$-flat closed subscheme $D \subset X_{T}$ whose ideal $I_{D}$ is invertible.  

If $D$ is a $T$-relative effective Cartier divisor, then we write $\calO_{X_T}(D)$ for $\Hom(I_D, \mathcal{O}_{X})$.

$\pi_1^{\ell}$ denotes the maximal pro-$\ell$ quotient of the \'etale fundamental group. $\pi_1^{(p')}$ denotes the maximal prime-to-$p$ quotient of the \'etale fundamental group, where a profinite group is said to be prime-to-$p$ if it can be expressed as an inverse limit of finite groups whose orders are not divisible by $p$. If two geometric points $a$ and $b$ are included in the notation, as in $\pi_1(-,a,b)$, this $\pi_1(-,a,b)$ denotes the set of \'etale paths from $a$ to $b$ i.e.,~the natural transformations from the fiber functor associated to $a$ to the fiber functor associated to $b$. If $\mathcal{E}$ is a set of geometric points, $\pi_1(-, \mathcal{E})$ denotes the fundamental groupoid based at $\mathcal{E}$. 

$R$ denotes $\Z_{\ell}$ or $\Z/\ell^n$.

\section{Construction of a one point compactification} \label{Section: FoldSingularities}
Here we define the one point compactification $\Xsing$ of a nonproper smooth curve $X$ over $k$.  We describe the structure on the compactified Picard scheme of $\Xsing$ in Section~\ref{Section: PresentationScheme} below.  For the remainder of this section, we fix a smooth and proper curve $\widetilde{X}$  over $k$ and a collection $\bound \subset \widetilde{X}$ of closed points with the property that $k(x)$ is a separable extension $k$ for all $x \in \bound$.  We will consider $\bound$ as a closed subscheme $j \colon \bound \to \widetilde{X}$ with the reduced scheme structure, and the separability assumption is equivalent to the assumption that $\bound$ is $k$-\'{e}tale.  

To $\widetilde{X}$, we associate the curve $\Xsing$ defined by the following pushout diagram
			\begin{equation}
				\begin{CD} \label{Diagram: Curvepush-out}
					\bound	 				@>j>> 				\widetilde{X} \\
					@VVV										@VfVV \\
					\Spec(k) 					@>\overline{x}_{0}>>	\Xsing
				\end{CD}
			\end{equation}
The pushout exists by \cite[Theorem~5.4]{ferrand03}.\hidden{

\begin{lm}
	Diagram (5.4.1) in \cite[Theorem~5.4]{ferrand03} is a pushout of schemes.
\end{lm}

\begin{proof}
Let $Z$ be a scheme. We need to show that the natural map $$\theta: \Mor_{sch} (X, Z) \to \Mor_{sch} (X', Z) \times_{\Mor_{sch} (Y', Z)} \Mor_{sch} (Y, Z)$$ is a bijection. Morphisms of schemes are the subset of morphisms of ringed spaces such that for all points $x$ with image $z$, the induced map of local rings $f_x: \mathcal{O}_z \to \mathcal{O}_x$ is local in the sense that $f_x^{-1}(m_x) = m_z$.  Thus the vertical arrows in the commutative diagram $$\xymatrix{ \Mor_{sch} (X, Z) \ar[r] \ar[d] & \ar[d] \Mor_{sch} (X', Z) \times_{\Mor_{sch} (Y', Z)} \Mor_{sch} (Y, Z) \\ \Mor_{ring-space} (X, Z) \ar[r] & \Mor_{ring-space} (X', Z) \times_{\Mor_{ring-space} (Y', Z)} \Mor_{ring-space} (Y, Z)}$$ are injective. By \cite[Theorem~5.4]{ferrand03}, the bottom horizontal arrow is a bijection. Thus $\theta$ is injective. Suppose that $(h,g)$ is an element of $ \Mor_{sch} (X', Z) \times_{\Mor_{sch} (Y', Z)} \Mor_{sch} (Y, Z)$. There exists $f: X \to Z$ a map of ringed spaces restricting to $(h,g)$. It remains to show that $f_x^{-1} (m_x) = m_z$. For $x \in X -Y$, the map $f$ equals $h$ on the open neighborhood $X-Y$ by \cite[Theorem~5.4(d)]{ferrand03}. Since $h$ is a map of schemes the equality holds. Otherwise, there exists $y \in Y$ with $u(y) = x$, and we have $$\xymatrix{g: \mathcal{O}_z \ar[r]^f & \mathcal{O}_x \ar[r]^u & \mathcal{O}_y}.$$ Since $f^{-1} (m_x) \supseteq g^{-1} (m_y) = m_z$, we have $f_x^{-1} (m_x) \supseteq m_z$. For any map of local rings, $f_x^{-1} (m_x) \subseteq m_z$, giving the desired equality.
\end{proof}

}  (Certainly $\widetilde{X}$  satisfies Condition AF because the curve is projective.)  Furthermore, $\widetilde{X}$ is smooth over $k$ and $f$ is finite by \cite[Proposition~5.6]{ferrand03}, so $f \colon \widetilde{X} \to \Xsing$ is the normalization map.  We call the singularity $x_0 := \overline{x}_{0}(0)$ an \textbf{ordinary fold singularity}.  We write $b(x_0) := \operatorname{rank}_{k} H^{0}(\bound, \calO_{\bound})$ for  the number of geometric \textbf{branches} of $X$ at $x_0$.

Diagram~\eqref{Diagram: Curvepush-out} remains a pushout diagram after tensoring with an arbitrary field extension $k'$ of $k$ by \cite[Lemma~4.4]{ferrand03}.  Since $\widetilde{X}_{k'}$ is $k'$-smooth, $\Xsing_{k'}$ is thus constructed from a $k'$-smooth curve by identifying  a finite collection of closed points with separable residue fields.   

For later use, we need a concrete description of  the local ring of $\Xsing$ at $x_0$.  Ferrand  constructs $\Xsing$ as the pushout in the category of ringed spaces.  As such a pushout, the structure sheaf $\calO_{\Xsing}$ is the pullback defined by
			$$
				\begin{CD} 
					f_{*} \calO_{\bound} 				@<j^{*}<< 				f_{*} \calO_{\widetilde{X}} \\
					@AAA										@Af^{*}AA \\
					k(x_0) 						@<\overline{x}_{0}^{*}<<	\calO_{\Xsing}.
				\end{CD}
			$$
			Equivalently, $\calO_{\Xsing}$ can be described by the exact sequence
			$$
				0 \to \calO_{\Xsing} \to f_{*}\calO_{\widetilde{X}} \to f_{*}\calO_{\bound}/k(x_0) \to 0.
			$$
			Taking the stalk at $x_0$, we get
			\begin{equation} \label{Eqn: RingInNormal}
				0 \to \calO_{\Xsing, x_0} \to \calO_{\widetilde{X}, f^{-1}(x_0)} \to \bigoplus_{f(\widetilde{x})=x_0} k(\widetilde{x})/k(x_0) \to 0.
			\end{equation}	
			Here $\calO_{\widetilde{X}, f^{-1}(x_0)}$ is the semilocalization of $\widetilde{X}$ at the closed subset $f^{-1}(x_0) \subset \widetilde{X}$.	 When $k=\kbar$, we have $k(\widetilde{x})=k$ for all $\widetilde{x} \in f^{-1}(x_0)$, and so $\oplus k(\widetilde{x})/k(x_0)$ is just a $k(x_0)$-module of rank $b(x_0)-1$.  
						
		We will occasionally need to describe the curves $\overline{Y}$ lying between $\Xsing$ and its normalization $\widetilde{X}$.  These curves are exactly the curves constructed by partitioning the points $\partial X$ into subsets and glueing each subset together.  More precisely, given a finite surjection $\bound \to \partial Y$, the two obvious pushout squares fit into a commutative diagram
			\begin{equation} \label{Eqn: Ypush-out}
				\begin{CD}
					\bound		@>j>>				\widetilde{X} \\
					@VVV						@VhVV \\
					\partial Y		@>i>>				\overline{Y} \\
					@VVV						@VgVV \\
					\Spec(k)		@>\overline{x}_0>>		\Xsing.
				\end{CD}
			\end{equation}
	The morphism $g$ is proper and birational, so $\overline{Y}$ lies between $\Xsing$ and $\widetilde{X}$, and every  curve lying between $\Xsing$ and $\widetilde{X}$ can be constructed in this manner.  Indeed, suppose that we are given a factorization $\widetilde{X} \stackrel{h}{\longrightarrow} \overline{Y} \stackrel{g}{\longrightarrow} \Xsing$ with $g$ a proper birational map.  The scheme $\partial Y := g^{-1}(x_0)$ naturally fits into the commutative diagram  \eqref{Eqn: Ypush-out}, and now the squares are pullback squares.  These squares are in fact also pushout squares.  To verify this, we can reduce to the affine case (as $h$ and $g$ are affine), in which case the claim follows by direct computation.   

\section{Construction of the presentation scheme} \label{Section: PresentationScheme}
Here we define and study the presentation scheme associated to a curve $\Xsing$ from Section~\ref{Section: FoldSingularities}.  Thus we fix  a smooth curve $\widetilde{X}$ over $k$ and a collection $\bound \subset \widetilde{X}$ of closed points with separable residue fields and then define  $\Xsing$ by the Pushout Diagram  \eqref{Diagram: Curvepush-out}. As in Section~\ref{Section: FoldSingularities}, we write $x_0 \in \Xsing$ for the unique singularity of $\Xsing$, $b(x_0) := \operatorname{rank}_{k} H^{0}(\bound, \calO_{\bound})$ for the number of geometric branches, and $f \colon \widetilde{X} \to \Xsing$ for the normalization of $\Xsing$.  

Our definition of the presentation scheme is motivated by the following observation: if $L$ is a line bundle on $\Xsing$, then the adjoint 
\begin{equation} \label{Eqn: MotivatePres}
	i_{\text{can}} \colon L \to f_* f^{*} L
\end{equation}
of the identity $f^{*} L \to f^{*} L$ is an inclusion whose cokernel is naturally isomorphic to 
\begin{equation} \label{Eqn: Cokernel}
	f^{*}L|_{\bound}/L|_{x_0} = (\bigoplus_{f(\widetilde{x})=x_0} k(\widetilde{x}) \otimes f^{*}L)/ k(x_0) \otimes L,
\end{equation}
which is a $k(x_0)$-module of rank $b(x_0)-1$.  This assertion is the exact sequence \eqref{Eqn: RingInNormal} when $L=\calO_{\Xsing}$ and is proven below in Lemma~\ref{Lemma: ClassifyOne} when $L$ is any line bundle.  With $i_{\text{can}}$ in mind, we make the following definition.

\begin{df}
	Let $T$ be a $k$-scheme.  A family of \textbf{presentations} of a family of rank $1$, torsion-free sheaves $I$ over $T$ is an injective homomorphism $i \colon I \to (f_{T})_{*}M$ from  $I$ to the direct image of a line bundle $M$ on $\widetilde{X}_{T}$ with the property that the cokernel is a locally free $\calO_{\{x_0\}\times T}$-module of rank $b(x_0)-1$.  A \textbf{presentation} is a family of presentations over $T=\Spec(k)$.
\end{df}

Presentations are functorial in the following sense.  Suppose that we are given a morphism $s \colon S \to T$ and a family of presentations $i \colon I \to (f_{T})_{*} M$ over $T$.  Because $f$ is finite, the base change homomorphism $\phi_{s} \colon s^{*} (f_{T})_{*}M \to (f_{S})_{*} s^{*} M$ is an isomorphism, and the resulting composition 
\begin{equation} \label{Eqn: PresFunctoriality}
	\phi_{s} \circ s^{*}i \colon s^{*}I \to (f_{S})_{*} s^{*}M
\end{equation}
is a family of  presentations over $S$.  Indeed, because $\Coker(i)$ is $T$-flat, the homomorphism $s^{*}i$, or equivalently the homomorphism \eqref{Eqn: PresFunctoriality}, is injective with cokernel equal to $s^{*}(\Coker(i))$ which is a locally free $\calO_{\{x_0\}\times S}$-module of rank $b(x_0)-1$. We will use this functoriality to define a functor, but first we need to put an equivalence relation on presentations.
\begin{df}
Two families of presentations $i \colon I \to (f_{T})_{*}M$ and $i' \colon I' \to (f_{T})_{*} M'$ over $T$ are said to be \textbf{equivalent} if there exist a line bundle $N$ on $T$ and isomorphisms $I \cong I' \otimes_{T} N$, $(f_{T})_{*}M \cong (f_{T})_{*}M' \otimes N$ that make the following diagram commute 
$$
	\begin{CD}
		I		@>i>>		(f_{T})_{*}M \\
		@VVV							@VVV \\
		I' \otimes_{T} N		@>i'>>		(f_{T})_{*}M' \otimes_{T} N.
	\end{CD}
$$
\end{df}
Using this definition of equivalence, we define the presentation functor.
\begin{df}
	The \textbf{presentation functor} $\Pres^{\sharp} \Xsing$ of $\Xsing$ is the functor $$\Pres^{\sharp} \Xsing \colon \text{$k$-Sch} \to \text{Sets}$$ defined as follows. Given a $k$-scheme $T$, we set $(\Pres^{\sharp} \Xsing)(T)$ equal to the set of equivalence classes of families of presentations of degree $-1$ sheaves over $T$.  Given a morphism $s \colon S \to T$ of $k$-schemes, we define $\Pres^{\sharp} s \colon (\Pres^{\sharp}\Xsing)(T) \to (\Pres^{\sharp}(\Xsing)(S)$ by the rule that sends $i$ to the presentation $\phi_{s} \circ s^{*}i$ in \eqref{Eqn: PresFunctoriality}.  
\end{df}

The presentation functor is represented by a Grassmannian bundle over the Picard scheme $\Jac^{-1} \widetilde{X}$ under suitable hypotheses.  Recall the relative Grassmannian $$\operatorname{Grass}(\calV,b) \to B$$ associated to e.g.,~a non-negative integer $b$, an algebraic $k$-scheme $B$, and a locally free $\calO_{B}$-module $\calV$ is a $k$-scheme $\operatorname{Grass}(\calV,b)$  whose $T$-valued points are pairs $(t,q)$ consisting of a $k$-morphism $t \colon T \to B$ and a surjective homomorphism $t^{*}\calV \to \calW$ onto a locally free $\calO_{T}$-module of rank $b$.

The exact relation between a Grassmannian bundle and the presentation functor is described by the following lemma.

\begin{lm} \label{Lemma: PresExists}
Assume  a universal line bundle  $M_{\text{uni}}$ exists on $\widetilde{X} \times \Jac^{-1} \widetilde{X} $, and let $$\calV := \left( ( f\times 1)_{*}(M_{\text{uni}}) \right)| \{ x_0 \} \times \Jac^{-1} \widetilde{X}.$$ Then the presentation functor is representable by the projective $k$-scheme $$\operatorname{Grass}(\calV, b(x_0)-1).$$
\end{lm}

\begin{rmk}
When $\Xsing$ is a nodal curve, this lemma was proven by Altman--Kleiman \cite[Proposition~9]{altman90}, and the following proof is closely modeled on their argument.
\end{rmk}

\begin{proof}
We construct a natural transformation $\operatorname{Grass}(\calV, b(x_0)-1) \to \Pres^{\sharp} \Xsing$, and then we construct the inverse transformation.  The main point is that  the cokernel of a family of presentations $i \colon I \to (f_{T})_{*} M$ can be written as $q \colon (f_{T})_{*} M \to (j_{T})_{*} \calW$ for some locally free sheaf $\calW$ on $\{ x_0 \} \times T$ of rank $b(x_0)-1$, and the rule that sends $i$ to the adjoint $q^{\#} \colon (f_{T})_{*} M |\{ x_0 \} \times T \to \calW$ essentially defines the isomorphism $\Pres^{\sharp} \Xsing \cong \operatorname{Grass}(\calV, b(x_0)-1)$. 

We construct $\operatorname{Grass}(\calV, b(x_0)-1) \to \Pres^{\sharp} \Xsing$  by exhibiting a family of presentations over $\operatorname{Grass}(\calV, b(x_0)-1)$.  Temporarily set $\operatorname{G}$ equal to $\operatorname{Grass}(\calV, b(x_0)-1)$, $g \colon \operatorname{G} \to \{ x_0 \} \times \Jac^{-1} \widetilde{X}$ equal to the structure morphism, $\calW_{\operatorname{G}}$ for the universal quotient, $q_{\text{uni}} \colon g^{*} \calV \to \calW_{\operatorname{G}}$ equal to the universal surjection, and $j \colon \{ x_0 \} \to X$, $j' \colon f^{-1}(x_0) \to \widetilde{X}$ equal to the inclusions.  

Consider the line bundle $M := (1\times g)^{*} M_{\text{uni}}$ on $ \widetilde{X} \times \operatorname{G}$.  The authors claim that there is a canonical isomorphism  $\phi_{\text{can}} \colon (f \times 1)_{*} M|{\operatorname{G} \times \{x_0\}} \cong g^{*} \calV$. Given the claim, the composition
$$
    (f \times 1)_{*} M| { \{ x_0 \} \times \operatorname{G} } \stackrel{\phi_{\text{can}}}{\longrightarrow} g^{*} \calV \stackrel{q_{\text{can}}}{\to} \calW_{\operatorname{G}}
$$
is adjoint to a family of presentations $(f \times 1)_{*} M \to (j \times 1)_{*} \calW_{\operatorname{G}}$  over $\operatorname{G}$, and this family defines the desired morphism  $\operatorname{G} \to \Pres^{\sharp} \Xsing$.

The existence of $\phi_{\text{can}} \colon (f \times 1)_{*} M|{ \{ x_0 \} \times \operatorname{G} } \cong g^{*} \calV$ follows from the cohomological flatness of $f$, which follows since $f$ is finite.  Cohomological flatness implies that the base change homomorphism 
\begin{equation} \label{Eqn: FirstBaseChange}
	(f \times 1)_{*} M|{ \{ x_0 \} \times \operatorname{G}} \cong ( f|_{f^{-1}(x_0)} \times 1)_{*} (j' \times g)^{*} M_{\text{uni}}
\end{equation}
is an isomorphism.  A second application of cohomlogical flatness shows that the sheaf on the right-hand side of Equation~\eqref{Eqn: FirstBaseChange} appears in another base change isomorphism
$$
	(j \times g)^{*} (f \times 1)_{*} M_{\text{uni}} 	\cong 	(f|_{f^{-1}(x_0)} \times 1)_{*} (j' \times g)^{*} M_{\text{uni}},
$$
and we have
\begin{align} \label{Eqn: SecondBaseChange}
	(1 \times f|_{f^{-1}(x_0)})_{*} (g \times j' )^{*} M_{\text{uni}} 	\cong& (g \times j)^{*} (f \times 1)_{*} M_{\text{uni}} \\
													\cong& g^{*} (1 \times j)^{*} (f \times 1)_{*} M_{\text{uni}} \notag \\
													\cong& g^{*} \calV \notag
\end{align}
We now define $\phi_{\text{can}}$ to be the composition of \eqref{Eqn: FirstBaseChange} and \eqref{Eqn: SecondBaseChange}.

To show that $\operatorname{G} \to \Pres^{\sharp} \Xsing$ is an isomorphism, we construct  the inverse natural transformation.  Thus suppose that $I \to (f_{T})_{*}M$ is  a family of presentations over a given $k$-scheme $T$.  By definition, the cokernel $(f_{T})_{*}M / I$ of the presentation can be written as $(j_{T})_{*}\calW$ for some locally free sheaf $\calW$ on $\{ x_0 \} \times T$ of rank $b(x_0)-1$.  We now construct a surjection from $t^{*} \calV$ to a sheaf constructed from $\calW$.

The line bundle $M$ defines a morphism  $t \colon T \to \Jac^{-1} \widetilde{X}$.  The pullback $(1 \times t)^{*} M_{\text{uni}}$ may not be isomorphic to $M$, but by \cite[5.6(i)]{altman80} there exists a line bundle $N$ on $T$ with the property that there exists an isomorphism $(1 \times t)^{*} M_{\text{uni}} \cong M \otimes N$.  (Here we are confusing $N$ with its pullback under $X \times T \to T$, and we will continue to do so for the rest of the proof.)  If we fix one such isomorphism $\alpha$, then we can consider the composition 
\begin{equation} \label{Eqn: AdjointConstruct}
	(i_{T})^{*} (f_{T})^{*} (1 \times t)^{*} M_{\text{uni}} \stackrel{(i_{T})_{*} (f_{T})_{*} \alpha}{\longrightarrow}	(i_{T})^{*} (f_{T})_{*} M \otimes N 	\stackrel{(q \otimes 1)^{\sharp}}{\longrightarrow} \calW \otimes N
\end{equation}
with  $(q \otimes 1)^{\sharp}$  the adjoint to $q \otimes 1 \colon (f_{T})_{*} M \otimes N \to (i_{T})_{*} \calW \otimes N$.  A third application of the cohomological flatness of $f$ shows that a suitable base change homomorphism defines an isomorphism
$$
	t^{*} \calV \cong (i_{T})^{*} (f_{T})_{*} (1 \times t)^{*} M_{\text{uni}},
$$
and the composition of this isomorphism with the homomorphism \eqref{Eqn: AdjointConstruct} is a surjection $t^{*} \calV \to \calW \otimes N$ with locally free quotient of rank $b(x_0)-1$.  To show that this construction defines a map $(\Pres \Xsing)(T) \to \operatorname{G}(T)$, we need to show that this surjection only depends on the equivalence class of the presentation $I \to (f_{T})_{*}M$.

Thus suppose that we are given a second $N'$ and a second isomorphism $\alpha' \colon (1 \times t)^{*} M_{\text{uni}} \cong M \otimes N'$.  Because $f_{*} M$ is simple \cite[Lemma~5.4]{altman80}, $\alpha' \circ \alpha^{-1} \colon M \otimes N \cong M \otimes N'$ must be of the form $1\otimes \beta$ for an isomorphism $\beta \colon N \cong N'$.  The isomophisms $ (i_{T})_{*} f_{T}^{*} 1 \otimes \beta \colon (i_{T})_{*} f_{T}^{*} M \otimes N \cong  (i_{T})_{*} f_{T}^{*} M \otimes N'$ and $\beta \otimes 1 \colon \calW \otimes N \cong \calW \otimes N'$ define an isomorphism between the quotient associated to $(N,\alpha)$ and the quotient associated to $(N', \alpha')$.  This shows that the construction from the previous paragraph defines a map $(\Pres \Xsing)(T) \to \operatorname{G}(T)$.  To complete the proof, we now simply observe that the maps $(\Pres \Xsing)(T) \to \operatorname{G}(T)$ and $ \operatorname{G}(T) \to (\Pres \Xsing)(T) $ are inverse to each other.
\end{proof}

Lemma~\ref{Lemma: PresExists} does not assert that $\Pres^{\sharp} \Xsing$ is representable for all $\Xsing$ because a universal line bundle $M_{\text{uni}}$ does not always exist.  When $M_{\text{uni}}$ fails to exist, we do not prove that  $\Pres^{\sharp} \Xsing$ is representable, but in Proposition~\ref{Prop: PresExists} below we prove that the associated \'{e}tale sheaf is representable.  Motivated by this, we make the following definition.

\begin{df}
	Define $\Pres^{\text{\'{e}t}} \Xsing$ to be the \'{e}tale sheaf associated to  $\Pres^{\sharp} \Xsing$.  A $k$-scheme $\Pres \Xsing$ that represents $\Pres^{\text{\'{e}t}} \Xsing$ is called the \textbf{presentation scheme}.
\end{df}

We prove that the presentation scheme exists by reducing to Lemma~\ref{Lemma: PresExists}, and we make the reduction by extending scalars. If $k' \supset k$ is a field extension, then $\Xsing_{k'}$ has an associated presentation functor because, as we observed in Section~\ref{Section: FoldSingularities}, $\Xsing_{k'}$ is described by a suitable pushout construction.  Comparing the two definitions of the presentation functor, we have:

\begin{lm} \label{Lemma: PresCommutes}
	The formations of  $\Pres^{\sharp} \Xsing$ and  $\Pres^{\text{\'{e}t}} \Xsing$ commute with field extensions $k' \supset k$.  
\end{lm}

\begin{pr} \label{Prop: PresExists}
The presentation scheme exists and is a projective $\Jac^{-1} \widetilde{X}$-scheme.
\end{pr}
\begin{rmk}
	 As with Lemma~\ref{Lemma: PresExists}, our proof is closely modeled on work of Altman--Kleiman \cite[Theorem~12]{altman90}.
\end{rmk}

\begin{proof}
	When $\Jac^{-1} \widetilde{X}$ admits a universal family of line bundles, the proposition is  Lemma~\ref{Lemma: PresExists}.  In general, there exists a finite separable extension $k' \supset k$ with the property that $\widetilde{X}_{k'} \to \Spec(k')$ admits a section and hence $\Jac^{-1} \widetilde{X}_{k'}$ admits a universal family of line bundles.  Thus the presentation scheme $\Pres \Xsing_{k'}$ exists.  By Lemma~\ref{Lemma: PresCommutes}, $\Pres \Xsing_{k'}$ represents $\Pres^{\text{\'{e}t}} \Xsing_{k'}$.  The $k'$-scheme $\Pres \Xsing_{k'}$ thus carries natural descent data, and to complete the proof, it is enough to show that this descent data are effective.
	
	Consider the morphism  $\Pres \Xsing_{k'} \to \Jac^{-1} \widetilde{X}_{k'} \times \Jacbar^{-1} \Xsing_{k'}$ that sends $i \colon I \to (f_{T})_{*}M$ to the pair $(M, I)$.  The descent data on $\Pres \Xsing_{k'}$ extend to descent data on this morphism.  Furthermore, the morphism has finite fibers by the proof of \cite[Lemma~8]{altman90} (or  Proposition~\ref{Prop: ClassificationOfJacbar} below).  Both $\Pres \Xsing_{k'}$ and $\Jac^{-1} \widetilde{X}_{k'} \times \Jacbar^{-1} \Xsing_{k'}$ are $k'$-proper, so $\Pres \Xsing_{k'} \to \Jac^{-1} \widetilde{X}_{k'} \times \Jacbar^{-1} \Xsing_{k'}$ must be finite.  Descent data for finite morphisms are always effective, so we can conclude that the morphism and hence  $\Pres \Xsing_{k'}$ descend to $k$.
\end{proof}

The presentation scheme admits the following two morphisms.
\begin{df}
	The morphism $\PtoPic \colon \Pres \Xsing \to \Jac^{-1} \widetilde{X}$ is defined by the rule that sends a family of presentations $i \colon I \to f_{*}M$ to the line bundle $M$.  The morphism $\PtoJbar \colon \Pres \Xsing \to \Jacbar^{-1} \Xsing$ is defined by the rule that sends $i \colon I \to f_{*}M$ to $I$.
\end{df}
Certainly the rule sending $i \colon I \to f_{*}M$ to $I$ defines a morphism from $\Pres \Xsing$ to the moduli space of all rank $1$, torsion-free sheaves, but we should explain why the rule defines a morphism into the closure $\Jacbar^{-1}\Xsing$ of the line bundle locus.  The presentation scheme $\Pres \Xsing$ is geometrically irreducible because, by Lemma~\ref{Lemma: PresExists}, $\PtoPic$ realizes $\Pres \Xsing$ as a projective bundle over $\Jac^{-1} \widetilde{X}$.  We can conclude that the image of $\PtoJbar$ is geometrically irreducible, and the image also  contains the line bundle locus because every line bundle $M$ admits a presentation --- the presentation $i_{\text{can}}$.  We can conclude that the image of $\Pres \Xsing$ is $\Jacbar^{-1} \Xsing$, and in particular $\PtoJbar$ maps to $\Jacbar^{-1} \Xsing$.

Next we construct some presentations.

\begin{df} \label{Def: AssoPresentation}
	Let $T$ be a $k$-scheme.  Given $\overline{y} \colon T \to f^{-1}(x_0)$, the image of $$\overline{y} \times 1 \colon T \to \widetilde{X}_{T}$$ is a Cartier divisor that we denote by $y \subset \widetilde{X}_{T}$.  Suppose that we are also given a line bundle $N$ on $\widetilde{X}_{T}$.  Then set 
	$$
		M := N \otimes \calO_{\widetilde{X}_{T}}(f^{-1}(x_0)_{T} - y).
	$$
	The divisor $f^{-1}(x_0)_{T} - y$ is effective, so there is a natural inclusion $i \colon (f_{T})_{*}N \to (f_{T})_{*} M$.  We define $(M,i)$ to be the \textbf{presentation associated to $(N,\overline{y})$}.  
\end{df}
To see that $(M,i)$ is a family of presentations, observe that on $\widetilde{X}_{T}$ we have the exact sequence
$$
	0 \to N \to M \to M|_{f^{-1}(x_0)_{T} - y} \to 0.
$$
Here $N \to M$ is the natural inclusion and $M \to M|_{f^{-1}(x_0)_{T} - y}$ is the natural restriction.  The direct image of $N \to M$  under $f_{T}$ is $i$, and so the cokernel of $i$ is $(f_{T})_{*} M|_{f^{-1}(x_0)_{T} - y}$.  This cokernel is a free $k(x_0)$-module of rank $b(x_0)-1$ because $f^{-1}(x_0)_{T} - y$ is a finite flat $k$-scheme of degree $b(x_0)-1$.

Because the associated presentation is a presentation, if $N$ is a line bundle on $\widetilde{X}$ and $f^{-1}(x_0)$ admits a $k$-valued point, then $f_{*}(N)$ admits a presentation: the presentation associated to $(N, \overline{y})$. We now classify all the sheaves that admit a presentation.

\begin{lm} \label{Lemma: ClassifyOne}
	Assume $k=\kbar$.  If $g \colon \overline{Y} \to \Xsing$ is a proper birational morphism out of a curve $\overline{Y}$ with at most one singularity, then $g_{*} N$ admits a presentation for every line bundle $N$ on $\overline{Y}$.
\end{lm}
\begin{proof}
 
Given $\overline{Y}$, factor the normalization map  $f \colon \widetilde{X} \to \Xsing$  as $\widetilde{X} \stackrel{h}{\longrightarrow} \overline{Y} \stackrel{g}{\longrightarrow} \Xsing$.  By the discussion at the end of Section~\ref{Section: FoldSingularities}, the curve $\overline{Y}$ can be constructed as the pushout of $\widetilde{X}$ and some subset $\partial Y$ as in Diagram~\eqref{Eqn: Ypush-out}.  Label the fiber $g^{-1}(x_0) = \{ y_0, y_1, \dots, y_n \}$ so that the points $y_1, \dots, y_n$ are not singularities.  

On $\overline{Y}$, we have the homomorphism $i_{\text{can}} \colon N \to g_{*} g^{*}N$ that is adjoint to the identity.  This homomomorphism is injective with cokernel equal to a $k(y_0)$-module of rank $b(y_0)-1$.  To see this, observe that $i_{\text{can}}$ is certainly an isomorphism away from $y_0$.  Thus the kernel of $i_{\text{can}}$ is supported on a proper subset of $\overline{Y}$, but this is only possible if the kernel is zero as $N$ is torsion-free, showing injectivity.  The cokernel is supported on $y_0$, so to compute it, we can pass to an open neighborhood of $y_0$ and hence assume $N$ is trivial.  When $N$ is trivial, the claim follows from the existence of the exact sequence \eqref{Eqn: RingInNormal}.  We now use $i_{\text{can}}$ as follows.

	Because the points $y_1, \dots, y_n$ are not singularities, the line bundle  $\calO_{\overline{Y}}(y_1+\dots+y_n)$ is  well-defined, and we set $j \colon N \to   N \otimes \calO_{\overline{Y}}(y_{1}+\dots+y_{n})$ equal to the natural inclusion and $M$ equal to $h^{*}(N \otimes \calO_{\overline{Y}}(y_{1}+\dots+y_{n}))$.  The homomorphism  
	\begin{equation} \label{Eqn: ConstructedPresentation}
		f_{*} h^{*}(j) \circ g_{*}(i_{\text{can}}) \colon g_{*} N \to f_{*}M
	\end{equation}
	is a presentation of $g_* N$.  To verify this, we need to show that the cokernel is a $k(x_0)$-module of rank $b(x_0)-1$. 
	
	We compute the cokernel using 
	\begin{equation} \label{Eqn: PushfowardToPresentation}
		h_{*} h^{*}(j) \circ i_{\text{can}} \colon N \to h_{*}M.
	\end{equation}
	The direct image of this homomorphism under $g_*$ is  \eqref{Eqn: ConstructedPresentation}.
	
	Temporarily set $Q$ equal to the cokernel of \eqref{Eqn: PushfowardToPresentation}.  The module $Q$ must be supported on $\{ y_{0}, \dots, y_{n} \}$ because \eqref{Eqn: PushfowardToPresentation} equals $i_{\text{can}}$ away from $y_1, \dots, y_n$.  For the same reason, the localization $Q_{y_{0}}$ is a $k(y_0)$-module of rank $b(y_0)-1$.  Away from $y_0$, the homomorphism \eqref{Eqn: PushfowardToPresentation} coincides with $j$, so the stalk of $Q$ at $y_i$ for $i=1, \dots, n$ is a rank $1$ module over $k(y_i)$.  Since $g_{*} Q$ is the cokernel of the homomorphism \eqref{Eqn: ConstructedPresentation}, we can conclude that this cokernel is a $k(x_0)$-module of rank
	\begin{align*}
		b(y_0)-1+1+\dots+1 	=& b(y_0)-1+n \\
						=& b(x_0)-1.
	\end{align*}

\end{proof}

Next we show that the only sheaves that admit a presentation are the sheaves appearing in the previous lemma.
\begin{lm} \label{Lemma: AdmitPres}
	Assume $k=\kbar$.  If $I$ is a rank $1$, torsion-free sheaf that admits a presentation, then there exists a proper birational morphism $g \colon \overline{Y} \to \Xsing$ out of a curve $\overline{Y}$ with at most one singularity and a line bundle $N$ on $\overline{Y}$ such that $I=g_{*} N$.  
\end{lm}
\begin{proof}
	Let  $I$ be given.  We first show that there exists $g \colon \overline{Y} \to \Xsing$ and $N_0$ such that the stalks $I_{x_0}$ and $(g_{*}N_0)_{x_0}$ are isomorphic.  We then modify $N_0$ to produce a suitable $N$.
		
	We  construct $\overline{Y}$ and $N_0$ as follows.  Fix a presentation $i \colon I \to f_{*} M$ of $I$ with quotient $q \colon f_{*} M \to Q$ and consider the induced map $q_{x_0} \colon M_{f^{-1}(x_0)} = (f_{*} M)_{x_0} \to Q_{x_0}$ on stalks.  By definition, $i_{x_0}(I_{x_0})$ is the kernel of $q_{x_0}$, but we can also compute this kernel directly.  The kernel must contain the product $\mathfrak{p}_{x_0} \cdot M_{f^{-1}(x_0)}$ with the maximal ideal $\mathfrak{p}_{x_0} \subset \calO_{\Xsing, x_0}$ because $Q$ is a $k(x_0) = \calO_{X, x_0}/\mathfrak{p}_{x_0}$-module.  The induced homomorphism 
	\begin{equation} \label{Eqn: qOnFiber}
		\overline{q} \colon k(x_0) \otimes M_{f^{-1}(x_0)} = \bigoplus_{f(\widetilde{x})=x_0} (k(\widetilde{x}) \otimes M) \to k(x_0) \otimes Q.
	\end{equation}
	on fibers is a surjection from a rank $b(x_0)$ vector space to a rank $b(x_0)-1$ vector space, so $\overline{q}$ has a rank $1$ kernel.  Pick $s_0 \in M_{f^{-1}(x_0)}$ mapping to a generator of this kernel.  The element $s_0$ itself lies in $i_{x_0}(I_{x_0})$, and $i_{x_0}(I_{x_0})$ is generated by $s_0$ together with $\mathfrak{p}_{x_0} \cdot M_{f^{-1}(x_0)}$.  By using the Prime Avoidance Lemma to modify an arbitrary $s_0$, we can assume that $s_0$ satisfies  $s_0 \notin \mathfrak{p}_{\widetilde{x}}^{2} \cdot M_{\widetilde{x}}$ for all $\widetilde{x} \in f^{-1}(x_{0})$.

	We now use $s_0$ to construct $\overline{Y}$ and $N_0$.  Define $\partial Y$ be the scheme that is the union of copies of $\Spec(k)$ labeled by the points  $\widetilde{x} \in f^{-1}(x_0)$ with $s_0(\widetilde{x}) = 0$ plus one additional copy of $\Spec(k)$.  There is a natural morphism $\partial X \to \partial Y$ that sends a $\widetilde{x}$ satisfying $s_0(\widetilde{x}) = 0$ to the point labeled by $\widetilde{x}$ and sends every $\widetilde{x}$ with $s_0(\widetilde{x}) \ne 0$ to the additional $\Spec(k)$.  With this  $\partial X \to \partial Y$, we define $\overline{Y}$ by Diagram~\eqref{Eqn: Ypush-out}.  Let $f= \widetilde{X} \stackrel{h}{\longrightarrow} \overline{Y} \stackrel{g}{\longrightarrow} \Xsing$ be the natural factorization.  The authors claim that an isomorphism $\phi_{x_0} \colon (g_{*}\calO_{\overline{Y}})_{x_0} = \calO_{\overline{Y}, g^{-1}(x_0)} \cong i_{x_0}(I_{x_0})$ is given by  $\phi_{x_0}(r) = r \cdot s_0$.  This rule defines a homomorphism $\phi_{x_0} \colon \calO_{\overline{Y}, g^{-1}(x_0)} \to M_{f^{-1}(x_0)}$ that is injective, as it is a nonzero homomorphism out of a rank 1, torsion-free module. Now let us show that the image is contained in $i_{x_0}(I_{x_0})$.
	
	The element $\phi_{x_0}(1)=s_0$ certainly lies in $i_{x_0}(I_{x_0})$, but we cannot immediately conclude that $\phi_{x_0}(r) \in i_{x_0}(I_{x_0})$ for all $r$ because $I_{x_0}$ is not (a priori) a $\calO_{\overline{Y}, g^{-1}(y_0)}$-module.  However, a computation shows
	\begin{align*}
		q(r \cdot s_0) 	=& \sum_{f(\widetilde{x})=x_0} \overline{q}( r \cdot s_0(\widetilde{x})) \\
					=& \sum_{f(\widetilde{x})=x_0} \overline{q}( r(h(\widetilde{x})) s_0(\widetilde{x})) \\
					=& \sum_{\widetilde{x} \in \partial Y} \overline{q}( r(y_0) s_0(\widetilde{x})) \\
					=& r(y_0) q(s_0) \\
					=& 0,
	\end{align*}
	i.e.,~that $\phi_{x_0}(r)$ lies in the kernel of $q$.  
	
	To conclude that $\phi_{x_0}$ is an isomorphism, we need to show surjectivity.  Certainly the image of $\phi_{x_0}$ contains $s_0$, so we need to show that the image also contains $\mathfrak{p}_{x_0} \cdot M_{f^{-1}(x_0)}$.  A typical generator of $\mathfrak{p}_{x_0} \cdot M_{f^{-1}(x_0)}$ is $r \cdot  s$ with $r \in \mathfrak{p}_{x_0}$ and $s \in M_{f^{-1}(x_0)}$.  Since $M_{f^{-1}(x_0)}$ is generically free of rank $1$, we can certainly write $r  \cdot s = r_0 \cdot s_0$ for a unique \emph{rational} function $r_0$.  We show that $r_0 \in \calO_{\overline{Y}, g^{-1}(x_0)}$ by examining the stalk at a point $\widetilde{x} \in f^{-1}(x_0)$.  
	
	If $\widetilde{x} \in \partial Y$, then  $M_{\widetilde{x}}$ is freely generated by $s_0$.  Writing $s =r_1 \cdot s_0$ for $r_1 \in \calO_{X, \widetilde{x}}$, we see $r_0 = r \cdot r_1$, so $r_0 \in \mathfrak{p}_{\widetilde{x}} \subset \calO_{\widetilde{X}, \widetilde{x}}$.  If $\widetilde{x} \notin \partial Y$, then $s_0$ does not generate $M_{\widetilde{x}}$, but the section does generate $\mathfrak{p}_{\widetilde{x}} \cdot M_{\widetilde{x}}$, and similar reasoning shows $r_0 \in \calO_{\widetilde{X}, \widetilde{x}}$.  We can conclude that $r_0 \in \cap \calO_{\widetilde{X}, \widetilde{x}} =  \calO_{\widetilde{X}, f^{-1}(x_0)}$ and $r_0(\widetilde{x})=0$ for $\widetilde{x} \in \partial Y$ or equivalently $r_0 \in \calO_{\overline{Y}, g^{-1}(x_0)}$.  This completes the first part of the proof, that $I_{x_0}$ is isomorphic to $(g_{*} N_0)_{x_0}$ for $N_0 = \calO_{\overline{Y}}$.

	We now modify $N_0$ as in \cite[Lemma~3.1]{kass12}.  We can extend the isomorphism $\phi_{x_0} \colon I_{x_0} \cong  (g_{*}N_{0})_{x_0} $ to an isomorphism $\phi_1$ over some open neighborhood of $x_0$.  The complement of that neighborhood is contained in some open subset over which there exists an isomorphism $\phi_2$ between the restrictions of $I$ and $g_{*}N_0$.  If we define $L$ to be the line bundle given by gluing the trivial line bundle to the trivial line bundle using the  $\phi_{2}^{-1} \circ \phi_{1}$, then the isomorphisms $\phi_1$ and $\phi_2$ define an isomorphism $L \otimes g_{*} N_0 \cong I$.  The line bundle $N= g^{*}L \otimes N_0$ then satisfies the conditions of the lemma.
	
\end{proof}

\begin{lm} \label{Lemma: ComputeRank}
	Assume $k=\kbar$.  Let $N$ be a line bundle on $\Xsing$, $M$ a line bundle on $\widetilde{X}$, and $i_{x_0} \colon N_{x_0} \to (f_{*}M)_{x_0}$ an injection whose cokernel is a $k(x_0)$-module.  If $\Xsing \neq \widetilde{X}$, then $\operatorname{rank}_{k(x_0)} (f_{*}M)_{x_0}/i_{x_0}(N_{x_{0}})=b(x_0)-1$.
\end{lm}
\begin{proof}
	We can pass to the local ring and hence replace $N$ with $\calO_{\Xsing, x_0}$ and $M$ with $\calO_{\widetilde{X}, f^{-1}(x_0)}$.  The homomorphism $i$ is then given by $r \mapsto f^{*}(r) \cdot s_0$ for some $s_0 \in \calO_{\widetilde{X}, f^{-1}(x_0)}$.  The authors claim that $s_0$ is a unit.
	
	We assume $s_0$ is not a unit and then derive a contradiction by constructing an element of $\Coker(i)$ not killed by the maximal ideal $\mathfrak{p}_{x_0}$ of $x_0$.  If not a unit, $s_{0}(\widetilde{x}_{0})=0$ for some $\widetilde{x}_{0} \in f^{-1}(x_0)$.  By the Prime Avoidance Lemma, we can pick a generator  $r_0$ of $\mathfrak{p}_{\widetilde{x}_{0}}$ with the property that $r_0(\widetilde{x}) \ne 0$ for $\widetilde{x} \in f^{-1}(x_0)$, $\widetilde{x} \ne \widetilde{x}_{0}$.  The element $s := r^{-1}_{0} \cdot s_0$ lies in $\calO_{\widetilde{X}, f^{-1}(x_0)}$, and its image in $\Coker(i)$ is not killed by any $t \in \mathfrak{p}_{x_0}$ satisfying $f^{*}(t) \notin \mathfrak{p}_{\widetilde{x}_0}^{2}$.  Indeed, if $t$ kills $s$, then $f^{*}(t_{0}) = r_{0}^{-1} \cdot f^{*}(t)$ for some $t_{0} \in \calO_{\Xsing, x_0}$.  Evaluating this equation at any  $\widetilde{x} \ne \widetilde{x}_0$, we see that $t_{0}(x_0)=0$, and by comparing orders of vanishing at $\widetilde{x}_{0}$, we see that this is only possible if $f^{*}(t) \in \mathfrak{p}_{\widetilde{x}_{0}}^{2}$.   This shows that $\Coker(i)$ is not a $k(x_0)$-module, so we have derived the desired contradiction.  We can conclude that $s_0$ is a unit, and so $\Coker(i)=\calO_{\widetilde{X}, f^{-1}(x_0)}/f^{*}(\calO_{\Xsing, x_0})$ visibly has $k(x_0)$-rank $b(x_0)-1$.  
\end{proof}

We now classify the presentations $i \colon I \to f_{*} M$ with $I$ and $M$ fixed.

\begin{lm}
	Assume $k=\kbar$.  Up to equivalence, a rank $1$, torsion-free sheaf $I$ admits at most one presentation unless $I=f_{*}N$ for $N$ a line bundle on $\widetilde{X}$.  When $I=f_{*}N$, there are exactly $b(x_0)$ inequivalent presentations: the associated presentations from Definition~\ref{Def: AssoPresentation}.
\end{lm}

\begin{proof}
	By the Lemma~\ref{Lemma: AdmitPres}, we can assume $I=g_{*}N$ for some line bundle $N$ on a  curve $\overline{Y}$ with at most one singularity.  Factor  $f$ as $\widetilde{X} \stackrel{h}{\to} \overline{Y}  \stackrel{g}{\to} \Xsing$ and label the fiber $g^{-1}(x_0) = \{ y_0, \dots, y_n \}$ so that the points $y_1, \dots, y_n$ are not singularities.  
	
	Suppose that $i$ is a given presentation.  We can write $i$ as  $i=g_{*}i'$ for $i' \colon N \to h_{*}M$ by \cite[Proposition~3]{altman90}.  If $\operatorname{ad}(i') \colon h^{*}N \to M$ is the adjoint to $i'$ and $i'_{\text{can}} \colon h^{*} h_{*}N \to N$ is the adjoint to the identity $h^{*}N \to h^{*}N$, then we have
	$$
		i' = h_{*}\operatorname{ad}(i') \circ i'_{\text{can}}.
	$$
	We first consider the case where $\overline{Y} \ne \widetilde{X}$, the case where we need to prove that $i$ is equivalent to the presentation constructed in Lemma~\ref{Lemma: ClassifyOne}.
	
	Consider the cokernel $Q'$ of $i'$.  The direct image $g_{*}Q'$ is the cokernel of $i$, so since $g_{*}Q'$ is a $k(x_0)$-module of rank $b(x_0)-1$, $Q'$ is a module over $k(y_0) \oplus k(y_1) \oplus \dots \oplus k(y_n)$ and the sum $\sum \operatorname{rank}_{k(y_i)} Q'_{y_i}$ is $b(x_0)-1$.  We have $\operatorname{rank}_{k(y_0)} Q'_{y_0} = b(y_0)-1$ by Lemma~\ref{Lemma: ComputeRank} and for $i \ne 0$, the $k(y_i)$-rank of $Q'_{y_i}$ is at most $1$ as this module is a quotient of the rank $1$ module $(h_{*}M)_{y_i}$.  Combining these inequalities, we get 
	\begin{align*}
		b(x_0)-1 =& \operatorname{rank}_{k(y_0)}Q_{y_0} + \dots + \operatorname{rank}_{k(y_n)}Q_{y_n} \\
			\le& b(y_0)-1 + 1 + \dots + 1 \\
			=& b(y_0)-1+n \\
			=& b(x_0)-1,
	\end{align*}
	so all the inequalities must be equalities.  In other words, $Q_{y_i}$ is a rank $1$ module over $k(x_i)$ for all $i \ne 0$, and $Q_{y_0}$ has rank $b(y_0)-1$.
	
	 The cokernel of $(i'_{\text{can}})_{y_0} \colon N_{y_0} \to (h_{*}h^{*}N)_{y_0}$ is a $k(y_0)$-module of rank $b(y_0)-1$ (this was a computation in Lemma~\ref{Lemma: ClassifyOne}), so $h_{*}(\operatorname{ad}(i'))_{y_0}$ must be an isomorphism.  Similar reasoning shows that $h_{*}(\operatorname{ad}(i'))_{y_i}$ has a rank $1$ kernel.  In other words,
	$$
		\operatorname{ad}(i') \colon h^{*}N \to M
	$$
	has cokernel isomorphic to $k(\widetilde{y}_{1}) \oplus \dots \oplus  k(\widetilde{y}_{n})$.  Equivalently $\operatorname{ad}(i')$ factors as 
	$$
		h^{*}N \to \calO_{\widetilde{X}}( \widetilde{y}_{1} + \dots + \widetilde{y}_{n}) \otimes h^{*}N \to M
	$$
	with $\calO_{\widetilde{X}}( \widetilde{y}_{1} + \dots + \widetilde{y}_{n}) \otimes h^{*}N \to M$ an isomorphism (by degree considerations).  This isomorphism defines an equivalence between the given presentation and the presentation from Lemma~\ref{Lemma: ClassifyOne}.
	
	To complete the proof, we must consider the case $\overline{Y}=\widetilde{X}$.  Most of the argument given in the previous case remains valid except that the rank of $k(y_{0}) \otimes M_{y_{0}}$ is $b(y_0)=1$, not $b(y_{0})-1=0$.  Thus we can only conclude that from the fact that $b(x_0)-1 = \sum \operatorname{rank}_{k(y_{i})} Q_{y_{i}}$ that the rank of $Q_{y_{i}}$ is $1$ for all but exactly one $i=i_0$, in which case $\operatorname{rank}_{k(y_{i_{0}})} Q_{y_{i_{0}}}=0$. The rest of the argument shows that the given presentation $(i,M)$ is equivalent to the presentation associated to $M$ and $y_{i_0}$.
	\end{proof}

The following proposition summarizes the past four lemmas.

\begin{pr} \label{Prop: ClassificationOfJacbar}
	Assume $k=\kbar$.  Then a rank $1$, torsion-free sheaf $I$ on $\Xsing$ admits a presentation $i \colon I \to f_* M$ if and only there exists a curve $\overline{Y}$ with at most one singularity, a proper birational morphism $g \colon \overline{Y} \to \Xsing$, and a line bundle $N$ on $\overline{Y}$ such that $I=g_{*}N$.
	
	Furthermore, the presentation is unique unless $\overline{Y}=\widetilde{X}$ is the normalization, in which case every presentation is isomorphic to the presentation associated to some $(M, \overline{y})$.  In particular, there are exactly $b(x_0)$ different presentations of $I$. 
\end{pr}

Using this proposition, we can immediately construct rank $1$, torsion-free sheaves that are not  limits of line bundles.  For example:

\begin{cor} \label{Cor: NotLimit}
	Assume $\operatorname{char}(k)>3$.  Define $\Xsing$ by the Pushout Diagram \eqref{Diagram: Curvepush-out} with $\widetilde{X} :=\bbP^{1}$ and $\{ \pm1, \pm2 \} := \partial X$.  Define $\overline{Y}$ by the Pushout Diagram \eqref{Eqn: Ypush-out} with  $\partial Y := \{ \pm1 \} \cup \{ \pm2 \}$ and $\partial X \to \partial Y$ the morphism $\pm 1 \mapsto 1$, $\pm 2 \mapsto 2$.  If $g \colon \overline{Y} \to \Xsing$ is the natural morphism, then
	$$
		I = g_{*} \calO_{\overline{Y}}
	$$
	is not the limit of line bundles.
\end{cor}
\begin{proof}
	If we fix a degree $1$ line bundle $L_0$ on $\Xsing$, then the map $\Pres \Xsing \to \Jacbar^{0} \Xsing $ that sends $i \colon I \to f_{*}M$ to $I \otimes L_0$ is surjective, and the image does not contain $[g_{*}\calO_{\overline{Y}}]$ by Proposition~\ref{Prop: ClassificationOfJacbar}.
\end{proof}

We now use the result just proven to describe the structure of $\Jacbar \Xsing $.  

\begin{df} \label{Def: Sections}
	Define the \textbf{canonical embedding into the presentation scheme} $$\widehat{\epsilon} \colon \Jac^{0} \widetilde{X} \times f^{-1}(x_0) \to \Pres \Xsing$$ by the rule that sends $(N, \overline{y}) \in (\Jac^{0} \widetilde{X})(T) \times f^{-1}(x_0)(T)$ to the associated presentation of $(N \otimes \mathcal{O}_{\widetilde{X}_{T}}(-\bound_{T}), \overline{y})$.
	
	The \textbf{canonical embedding into the compactified Picard scheme} $$\epsilon \colon \Jac^{0} \widetilde{X} \to \Jacbar^{-1} \Xsing$$ is defined by the rule $L \mapsto f_{*} (L \otimes \mathcal{O}_{\widetilde{X}_{T}}(-\bound_{T}))$. 
\end{df} 
The canonical embedding into the compactified Picard scheme is a closed embedding by \cite[Section~5, p.~19]{altman90} (which deduces this property by using \cite[Proposition~3]{altman90} to assert that $\epsilon$ is a monomorphism), and the same argument shows that $\widehat{\epsilon}$ is a closed embedding.

Proposition~\ref{Prop: ClassificationOfJacbar} allows us to describe $\Jacbar^{-1} \Xsing$ up to universal homeomorphism as follows. By \cite[Theorem~5.4]{ferrand03}, we may define $\overline{P}^{\natural}$ by the pushout diagram 
\begin{equation}\label{defineJnat}
	\begin{CD}
		f^{-1}(x_0) \times \Jac^{0} \widetilde{X}		@>\widehat{\epsilon}>>		\Pres \Xsing \\
		@VVV									@VVV \\
		\Jac^{0} \widetilde{X} 					@>>>		\overline{P}^{\natural}
	\end{CD}
\end{equation} 
where $f^{-1}(x_0) \times \Jac^{0} \widetilde{X} \to \Jac^{0} \widetilde{X}$ is the projection morphism. Because the diagram 
$$
	\begin{CD}
		f^{-1}(x_0) \times \Jac^{0} \widetilde{X}		@>\widehat{\epsilon}>>		\Pres \Xsing \\
		@VVV									@V\PtoJbar VV \\
		\Jac^{0} \widetilde{X} 					@>\epsilon>>		\Jacbar^{-1} \Xsing
	\end{CD}
	$$
	commutes, the universal property of the pushout defines a map $\overline{P}^{\natural} \to\Jacbar^{-1} \Xsing$.

\begin{tm} \label{Proposition: UnivesalHomeo}
The map $\overline{P}^{\natural} \to\Jacbar^{-1} \Xsing$ is a universal homeomorphism.
\end{tm}
\begin{proof}
	The morphism $\overline{P}^{\natural} \to \Jacbar^{-1} \Xsing$ is surjective because, as we observed after the construction of $\Pres \Xsing$, the morphism $\Pres \Xsing \to \Jacbar^{-1} \Xsing$ is surjective.  Because both $\Pres \Xsing$ and $\Jacbar^{-1} \Xsing$ are $k$-proper, it is enough to show that the induced map $\overline{P}^{\natural}(\kbar) \to (\Jacbar^{-1} \Xsing)(\kbar)$ on $\kbar$-points is injective \cite[Proposition~2.4.5]{egaIV_2}. 
	
	We can describe the $\kbar$-valued points of $\overline{P}^{\natural}$ explicitly.  Applying the functor taking a scheme to the set of its $\kbar$-points sends \eqref{defineJnat} to a pushout diagram, as can be verified e.g.,~by using the fact that $\Pres \Xsing \to \Jacbar^{-1} \Xsing$ is isomorphism away from the image of the canonical embedding and that \eqref{defineJnat} is a pullback diagram \cite[Theorem~5.4]{ferrand03}. Because the formation of the presentation scheme commutes with field extensions, $(\Pres\Xsing)(\kbar) = (\Pres \Xsing_{\kbar})(\kbar)$, so $\overline{P}^{\natural}(\kbar) \to (\Jacbar^{-1} \Xsing)(\kbar)$ is bijective by Proposition~\ref{Prop: ClassificationOfJacbar}.
\end{proof}

\section{The Abel map}\label{Section_Abel_map}
Here we study the Abel map of the curve $\Xsing$ from Section~\ref{Section: FoldSingularities}.  We first compare the definitions of two different Abel maps associated to a smooth curve over $k$ and show that only one of them naturally extends to $\Xsing$.  Second, we lift the Altman--Kleiman Abel map $\Abel \colon \Xsing \to \Jacbar^{-1} \Xsing$ to a morphism $\widetilde{X} \to \Pres \Xsing$ from the normalization to the presentation scheme.  In Section~\ref{section_H1Jbar} we will use this lifted Abel map to compute $H_{1}(\Abel)$.

The Abel map of a smooth curve $\widetilde{X}$ over $k$ is defined in standard texts such as \cite{mumford75}  to be the map $\widetilde{X} \to \Jac^{1}\widetilde{X}$ that sends  a point $x$ to the line bundle $\calO_{\widetilde{X}}(x)$, which is the dual $\calO_{\widetilde{X}}(x) := \underline{\operatorname{Hom}}(I_{x}, \calO_{\widetilde{X}})$ of the ideal $I_{x}$ of $x$.  We denote this morphism by $\Abel^{\vee} \colon \widetilde{X} \to \Jac^{1} \widetilde{X}$ to distinguish it from the morphism $\Abel \colon \widetilde{X} \to \Jac^{-1} \widetilde{X}$ studied by Altman--Kleiman in \cite{altman80}.  There the authors define the Abel map $\Abel \colon \widetilde{X} \to \Jac^{-1} \widetilde{X}$ to be the map that sends a point $x$ to its ideal sheaf $I_{x}$.  

Altman--Kleiman prove that this second definition extends to singular curves: given a singular curve $\overline{X}$ , the  rule $x \mapsto I_{x}$ defines a morphism $\Abel \colon \overline{X} \to \Jacbar^{-1} \overline{X}$ that is a closed embedding provided the genus of $\overline{X}$ is at least $1$ \cite[Theorem~8.8]{altman80}.  We call this map the \textbf{Altman--Kleiman Abel map} of $\overline{X}$.  

The first definition, the definition of $\Abel^{\vee}$, does not always extend to singular curves.  This issue is slightly subtle.  To show that  the rule 
\begin{equation} \label{Eqn: BadAbelRule}
	x \mapsto \underline{\Hom}(I_{x}, \calO_{\overline{X}})
\end{equation}
 defines a morphism $\overline{X} \to \Jacbar^{1} \overline{X}$, one must show that a flat deformation of $x$ induces a flat deformation of $\underline{\Hom}(I_{x}, \calO_{\overline{X}})$.  That is, given a $k$-morphism $x(t) \colon T \to \overline{X}_{T}$, one must show that the duals $\underline{\Hom}(I_{x(t_0)}, \calO_{\overline{X}})$ of the ideals of the fibers $x(t_0)$ of $x(t)$ fit together to form the fibers of a $\calO_{T}$-flat family of sheaves on $\overline{X}_{T}$.  
 
 When $x(t)$ maps into the smooth locus $X := \overline{X}^{\text{sm}}$, essentially the same construction as in the smooth case produces a suitable family of sheaves, and so Equation~\eqref{Eqn: BadAbelRule} defines a morphism $\Abel^{\vee} \colon X \to \Jac^{1} \overline{X}$ from the smooth locus to the Picard scheme.  We call this morphism the \textbf{classical Abel map}.    When $\overline{X}$ is Gorenstein, the classical Abel map extends to a morphism defined on all of $\overline{X}$ because a construction using cohomology and base change produces a suitable family for an arbitrary $x(t)$.  (For the construction, see \cite[Definition~5.0.7]{kass13}.  The argument is a modification of \cite[2.2]{esteves00}.)
 
 The one point compactification $\Xsing$ is, however, Gorenstein only when $b(x_0)=2$.  When $b(x_0) \ge 3$, not only does the construction just reviewed fail to produce a morphism $\Xsing \to \Jacbar^{1} \Xsing$ extending the classical Abel map $\Abel^{\vee}$, but it may be impossible to extend the classical Abel map to a morphism  out of $\Xsing$ by any construction.  We demonstrate this with an example.

\begin{exa} \label{ex_Abeldual_can_fail_to_extend} The classical Abel map $\Abel^{\vee} \colon X \to \Jac^{1} \Xsing$ can fail to extend to a morphism $\Abel^{\vee} \colon \Xsing \to \Jacbar^{1} \Xsing$.
\end{exa}
\begin{proof}[Proof of example]
	Let $k$ be a field of characteristic not $2$.  Define $\Xsing$ by the Pushout Diagram \eqref{Diagram: Curvepush-out} with  $\widetilde{X} := \bbP^{1}$ and $\bound :=\{ 0, 1,-1 \}$, so $\Xsing$ is a rational curve with a unique singularity $x_0$ that is an ordinary $3$-fold point.  We will show that the classical Abel map $\Abel^{\vee}$ of $\Xsing$ is undefined at $x_0$.
	
	We argue as follows.  The composition $\Abel^{\vee} \circ f \colon \widetilde{X} - f^{-1}(x_0) \to \Jacbar \Xsing$ extends to a regular map out of $\widetilde{X}=\bbP^{1}$ by the valuative criterion, and we show directly that the points $\Abel^{\vee}\circ f(0)$, $\Abel^{\vee} \circ f(1)$, and $\Abel^{\vee} \circ f(-1)$ are  distinct.  More precisely 
	\begin{align} \label{Eqn: ThreePoints}
		\Abel^{\vee} \circ f(0) =& [ (g_0)_{*}\calO_{\overline{Y}_0} ], \\
		\Abel^{\vee} \circ f(1) =& [ (g_1)_{*}\calO_{\overline{Y}_1} ], \notag \\
		\Abel^{\vee} \circ f(-1) =& [ (g_{-1})_{*}\calO_{\overline{Y}_{-1}}],  \notag
	\end{align}
	where $g_{0} \colon \overline{Y}_{0} \to \Xsing$, $g_{-1} \colon \overline{Y}_{-1} \to \Xsing$, and $g_{1} \colon \overline{Y}_{1} \to \Xsing$ are the three nodal curves lying between $\Xsing$ and $\widetilde{X}$. So, $\overline{Y}_{0}$ is the pushout of $\widetilde{X}=\bbP^{1}$ and $\{1, -1\}$ over $\Spec(k)$, etc.
	
	We verify Equation~\eqref{Eqn: ThreePoints} for the point $0$ and leave the remaining cases (which involve only notational changes to the argument given) to the interested reader. To verify the equation, we use the dualizing sheaf $\omega$.  Recall the rule $I \mapsto \underline{\Hom}(I, \omega)$ defines an involution on the set of rank $1$, torsion-free sheaves and the formation of $\underline{\Hom}(I, \omega)$ commutes with families  \cite[2.2]{esteves00}.  We verify Equation~\eqref{Eqn: ThreePoints}  by constructing a family of sheaves such that the dual family defines $\Abel^{\vee} \circ f$.

	The dualizing module $\omega$ can be described as the module of Rosenlicht differentials, i.e.~rational differentials on $\widetilde{X}$ with at worst simple poles at $0, -1,$ and $1$ and satisfying the condition that the residues at these points sum to zero.  Because a section of $\omega$ has at worst a simple pole at $0$, the rule $\phi(p(t)/q(t) dt)=a p(a)/q(a)$ defines a homomorphism 
	$$
		\phi \colon \omega \otimes k[a][1/(a^2-1)] \to k[a][1/(a^2-1)].
	$$
	Let $I_a$ equal the kernel. Note that away from $a = 0$, this kernel consists of those differentials in $\omega$ which vanish at $t=a$. At $a=0$, this kernel consists of such differentials whose residue at $0$ is $0$. In other words, for $\alpha \ne 0, -1, 1$, the fiber $I_{\alpha}$ of $I_{a}$ over $\Spec(k[a]/(a-\alpha)) \to \Spec(k[a][1/(a^2-1)])$ is $\omega(-f(a))$, and  $I_0$ is the sheaf of rational differentials on $\widetilde{X}$ with at worst simple poles at $-1$ and $1$ and satisfying the condition that the residues sum to zero.  By Rosenlicht's definition (or description) of $(g_{0})_{*} \omega_{Y_{0}}$, we see that $I_0 = (g_{0})_{*} \omega_{Y_{0}}$.   
	
	The kernel $\Ker \phi = I_a$ is $k[a][1/(a^2-1)]$-flat because $\phi$ is surjective.  
	
	Furthermore, the $\omega$-dual $\underline{\Hom}(I_{a}, \omega)$ defines $\Abel^{\vee} \circ f$.  Indeed, $\underline{\Hom}(I_{a}, \omega)$ is flat and its formation commutes with passing to fibers by \cite[2.2]{esteves00}.  Since $I_a =\omega(-f(a))$ for $\alpha \ne 0, -1, 1$, we deduce 
	$$
		\Abel^{\vee} \circ f(0) = [\underline{\Hom}(I_{0}, \omega)].
	$$ 
	
	 By the above, $I_0 =(g_{0})_{*} \omega_{Y_{0}} $.  The morphism
	$$
		(g_{0})_{*} \calO_{Y_0} \to \underline{\Hom}(g_{*} \omega_{Y_0}, \omega)
	$$
	defined by sending a rational function to multiplication by that function defines an isomorphism, showing that $\Abel^{\vee} \circ f(0) = (g_{0})_{*} \calO_{Y_0}$ as claimed.
\end{proof}

While the classical Abel map $\Abel^{\vee}$ may not be defined on all of $\Xsing$, the Altman--Kleiman Abel map $\Abel$ is not only defined on all of $\Xsing$, but it admits a natural lift to a morphism $\AbelLift \colon \widetilde{X} \to \Pres X$. Given $x=f(\widetilde{x})$ with $\widetilde{x} \in X \subset \widetilde{X}$,  the ideal sheaf $I_{x}$ of $x$ in $\Xsing$ admits the presentation $i_{\text{can}} \colon I_{x} \to f_{*} I_{\widetilde{x}} = f_{*} f^{*} I_{x}$ defined in \eqref{Eqn: MotivatePres}.  The resulting morphism $X \to \Pres \Xsing$ extends to the morphism $\AbelLift \colon \widetilde{X} \to \Pres \Xsing$ that we now define.

\begin{df}
	Define $\Gamma \subset \Xsing \times \widetilde{X}$ to be the transpose of the graph of $f$ and $\widetilde{\Delta} \subset \widetilde{X} \times \widetilde{X}$ to be the diagonal.  
	Define $I_{\widetilde{x}(t)}$ to be the ideal of $\widetilde{\Delta}$ and $J_{t}$ the ideal of $\Gamma$.
\end{df}

\begin{lm}\label{Jt_to_I_isprensentation}
	The natural inclusion $i \colon J_{t} \to (f \times 1)_{*} I_{\widetilde{x}(t)}$ is a family of presentations over $\widetilde{X}$.
\end{lm}
\begin{proof}
	We prove that $i$ is family of presentations by showing that  $\Coker(i)$ is isomorphic to  $(\bigoplus k(\widetilde{x}) )/k(x_0) \otimes \calO_{\widetilde{X}}$.  This last module is the cokernel of the inclusion $\calO_{\Xsing \times \widetilde{X}} \to \calO_{\widetilde{X} \times \widetilde{X}}$, as the formation of the pushout $\Xsing$ commutes with the flat base change $\Xsing \times \widetilde{X} \to \Xsing$.  All of the modules in question fit into the following commutative diagram with exact rows and columns:
	$$
		\begin{CD}
					@.		0							@.	0										@.	0		@.	 \\
			@.				@VVV								@VVV											@VVV			@. \\
			0		@>>>		J_{t}							@>>>	(f \times 1)_{*}I_{\widetilde{x}(t)}							@>>>	\Coker(i)	@>>>	0 \\
			@.				@VVV								@VVV											@VVV			@. \\
			0		@>>>		\calO_{\Xsing \times \widetilde{X}}	@>>>	(f \times 1)_{*} \calO_{\widetilde{X} \times \widetilde{X}} @>>>	(\bigoplus k(\widetilde{x}))/k(x_0) \otimes \calO_{\widetilde{X}} 	@>>>	0	\\
			@.				@VVV								@VVV											@VVV			@. \\
			0		@>>>		\calO_{\Gamma}				@>>>	(f \times 1)_{*}\calO_{\widetilde{\Delta}}			@>>>	C		@>>>	0 \\
			@.					@VVV								@VVV											@VVV			@. \\
					@.			0							@.		0										@.		0.		@.		
		\end{CD}
	$$
	The existence of the diagram follows by e.g.,~applying the Snake Lemma to the first two columns.  In the diagram, $\calO_{\Gamma}$ and $\calO_{\widetilde{\Delta}}$ are structure sheaves, and $C$ is the cokernel of $\Coker(i) \to (\bigoplus k(\widetilde{x}))/k(x_0)$ or equivalently of $\calO_{\Gamma} \to (f \times 1)_{*} \calO_{\widetilde{\Delta}}$.  
	
	Now $\Gamma$ and $\widetilde{\Delta}$ are both images of $\widetilde{X}$ under a closed embedding, and under the associated identifications $\Gamma = \widetilde{X}$, $\widetilde{\Delta} = \widetilde{X}$, the morphism $f \times 1 \colon \widetilde{\Delta} \to \Gamma$ becomes identified with the identity.  In particular, the natural morphism $\calO_{\Gamma} \to (f \times 1)_{*}\calO_{\widetilde{\Delta}}$ is an isomorphism and so $C=0$.  We can conclude that the natural morphism $\Coker(i) \to (\bigoplus k(\widetilde{x}))/k(x_0) \otimes \calO_{\widetilde{X}}$ is an isomorphism.  The target of this isomorphism is a free $\calO_{ \{ x_0 \} \times \widetilde{X}}$-module of rank $b(x_0)-1$, so we can conclude that $i$ is a family of presentations.
\end{proof}

\begin{df} \label{df_lifted_Abel_map}
	We define the \textbf{lifted Abel map} $\AbelLift \colon \widetilde{X} \to \Pres \Xsing$ to be the morphism defined by the family of presentations from Lemma~\ref{Jt_to_I_isprensentation}. 
\end{df}
By construction, the lifted Abel map has the property that the composition $\PtoPic \circ  \AbelLift$ is the Altman--Kleiman Abel map $\AbelTilde$ of $\widetilde{X}$, and  composition $\PtoJbar \circ \AbelLift$ is the composition $\Abel \circ f$ of Altman--Kleiman Abel map of $\Xsing$ with the normalization map.

\section{Cohomology of $X$ via $\Jac$}\label{Section:CohXfromPic}

Let $X$ be a smooth curve over $k$, contained as an open subset of $\widetilde{X}$, which is smooth and proper. Assume that $\widetilde{X}$ has genus greater than $0$, so its Abel map is non-trivial. We show that $\rH^1(X_{\kbar}, R(1))$ with its $\Gal(\kbar/k)$-module structure is obtained from the fundamental groupoid of $\Jac^{-1} \widetilde{X}$. When $\widetilde{X}$ has genus $0$, a degenerate form of this result holds, given in Remark~\ref{H1for_genus0_tildeX}.

For $R= \Z/\ell^m$ (respectively $\Z_{\ell}$), let  $\mathcal{U}_{R}$ denote the functor taking a finitely generated $R$-module to the underlying (topological) groupoid. Under suitable finiteness hypotheses on the category of groupoids, say groupoids with finitely many objects and finitely (topologically) generated morphism spaces, $\mathcal{U}_{R}$ has a left adjoint, denoted $\mathcal{F}_R$.

For such a groupoid $\pi$, with objects $D$ and source, target maps $s,t: \pi \to D$ respectively, there is a canonical functorial exact sequence \begin{equation}\label{groupoidSES} 0 \to \mathcal{F}_R \pi_1 \to \mathcal{F}_R \pi \stackrel{t-s}{\to} \oplus_D R \stackrel{\oplus_D \operatorname{id}} \to R \to 0,\end{equation}  where $\pi_1$ is the sub-groupoid spanned by any single object of $D$. To see that \eqref{groupoidSES} is canonical, note that the group $\mathcal{F}_R \pi_1$ is independent of the choice of object defining $\pi_1$ because a morphism $d_s \to d_t$ in $\pi$ determines an isomorphism $\pi_{d_s} \to \pi_{d_t}$ between the corresponding sub-groupoids, and applying $\mathcal{F}_R$, gives an isomorphism $\mathcal{F}_R \pi_{d_s} \to \mathcal{F}_R \pi_{d_t}$ independent of the choice of morphism from $d_s$ to $d_t$. 

Let $\bound \subsetneq \widetilde{X}$ be a closed subscheme with open complement $X$. For each point $\widetilde{x}$ of $\bound_{\kbar}$ choose a $\kbar$-geometric point with image $\widetilde{x}$ and let $\bound(\kbar)$ be the set of these chosen geometric points.  We may assume that $\bound(\kbar)$ is stable under the $\Gal(\kbar/k)$-action.

\begin{tm} \label{Prop_groupoid_to_H1}
	There is a canonical functorial isomorphism of $\Gal(\kbar/k)$-modules $$\mathcal{F}_R \pi_1^{\ell} (\Jac^{-1} \widetilde{X}_{\kbar}, \Abel_* \bound(\kbar)) \cong \rH^1(X_{\kbar}, R(1)).$$
\end{tm}

\begin{rmk}
	When $\bound= \emptyset$ one may take $\Abel_* \bound(\kbar)$ to be the set consisting of any single geometric point of $\Jac^{-1}\widetilde{X}_{\kbar}$.
\end{rmk}

\begin{rmk}\label{H1for_genus0_tildeX}
	When $\widetilde{X}$ has genus $0$, we have a canonical functorial short exact sequence of $\Gal(\kbar/k)$-modules $$\xymatrix{0 \ar[r] & \rH^1(X_{\kbar}, R(1)) \ar[r] & \oplus_{\bound(\kbar)} R \ar[rr]^{\oplus_{\bound(\kbar)} \operatorname{id}} && R \ar[r] & 0.}$$ This follows from applying cohomology to the pair $(X_{\kbar}, \widetilde{X}_{\kbar})$. Compare with \eqref{groupoidSES} and Theorem~\ref{Prop_groupoid_to_H1}.
\end{rmk}

\begin{proof}
Multiplication by $\ell^n$ $$m_{\ell^n}: \Jac^0 \widetilde{X}_{\kbar} \to \Jac^0 \widetilde{X}_{\kbar}$$ is a finite \'etale map. For a geometric point $a$ of $ \Jac^0 \widetilde{X}_{\kbar}$, let $\phi_a [m_{\ell^n}]$, abbreviated $\phi_a[\ell^n]$, denote the fiber above $a$.

For every $n$ and $\widetilde{x}_s$, $\widetilde{x}_t$ in $\bound(\kbar)$, there is a canonical, functorial\hidden{, Galois-equivariant (for the Galois group of any field over which $\widetilde{x}_s$ and $\widetilde{x}_t$ are defined)} map \begin{equation}\label{path_to_root} \pi_1^{\ell} (\Jac^{-1} \widetilde{X}_{\kbar}, \Abel_* \widetilde{x}_s, \Abel_* \widetilde{x}_t ) \to \phi_{\Abel_* \widetilde{x}_t - \Abel_* \widetilde{x}_s} [\ell^n]\end{equation} constructed as follows.

Addition by a $\kbar$-point $a$ of $\Jac^1 \widetilde{X}_{\kbar}$ defines $+_a: \Jac^{-1} \widetilde{X}_{\kbar} \to \Jac^0 \widetilde{X}_{\kbar}$. Pulling back $m_{\ell^n}$ by $+_a$ produces a finite \'etale map to $\Jac^{-1} \widetilde{X}_{\kbar}$, whence a map $$ \pi_1^{\ell} (\Jac^{-1} \widetilde{X}_{\kbar}, \Abel_* \widetilde{x}_s, \Abel_* \widetilde{x}_t ) \to \Mor (\phi_{ \Abel_* \widetilde{x}_s} [+_a^* m_{\ell^n}], \phi_{ \Abel_* \widetilde{x}_t} [+_a^* m_{\ell^n}]).$$ Since elements of $\pi_1^{\ell}$ are natural transformations between fiber functors, the image is contained in the morphisms $$\phi_{a+ \Abel_* \widetilde{x}_s}[\ell^n] \cong \phi_{ \Abel_* \widetilde{x}_s} [+_a^* m_{\ell^n}] \to \phi_{ \Abel_* \widetilde{x}_t} [+_a^* m_{\ell^n}]) \cong \phi_{a+ \Abel_* \widetilde{x}_t}[\ell^n]$$ given by addition by an element of $\phi_{\Abel_* \widetilde{x}_t - \Abel_* \widetilde{x}_s} [\ell^n]$. Sending the morphism to this element of $\phi_{\Abel_* \widetilde{x}_t - \Abel_* \widetilde{x}_s} [\ell^n]$ defines \eqref{path_to_root}, which is independent of the choice of $a$. 

By the moduli definition of the Picard functor, the $\kbar$-points of $\Jac \widetilde{X}_{\kbar}$ determine invertible sheaves on $\widetilde{X}_{\kbar}$. For the image $\widetilde{x}$ in $\widetilde{X}_{\kbar}$ of any element of $\bound(\kbar)$, the restriction of $\mathcal{O}_{\widetilde{X}_{\kbar}}(\widetilde{x})$ to $X_{\kbar}$ has a canonical trivialization coming from the inclusion of the ideal sheaf $I_{\widetilde{x}} =   \mathcal{O}_{\widetilde{X}_{\kbar}}(-\widetilde{x})$ into $\mathcal{O}_{\widetilde{X}_{\kbar}}$ which becomes an isomorphism after pullback to $X_{\kbar}$. Thus elements of $\phi_{\Abel_* \widetilde{x}_t - \Abel_* \widetilde{x}_s} [\ell^n]$ determine an invertible sheaf $L$ on $X_{\kbar}$ equipped with a trivialization of $\otimes^{\ell^n} L$.

The data of an invertible sheaf $L$ on $X_{\kbar}$ equipped with a trivialization of $\otimes^{\ell^n} L$ produces a canonical element of $\rH^1(X_{\kbar}, \Z/\ell^n(1))$ as follows:  there is a map $L \to \otimes^{\ell^n} L$ taking a section $a$ of $L$ over a Zariski-open $U$ to $\otimes^{\ell^n} a$. The coherent sheaves $L$ and $\otimes^{\ell^n} L$ determine $\G_m$-torsors, which we view as sheaves of sets in the \'etale topology equipped with an action of $\G_m$. Under this identification $a \mapsto \otimes^{\ell^n} a$ respects the actions of $\G_m$ in the sense that for $b \in \G_m$, $$a\cdot b \mapsto   \otimes^{\ell^n} a \cdot b^{\ell^n}.$$ There is an element $1$ in the set of sections over any \'etale open of the $\G_m$-torsor $\otimes^{\ell^n} L$ determined by the chosen trivialization of $\otimes^{\ell^n} L$. The sections $s$ of the $\G_m$-torsor $L$ such that $\otimes^{\ell^n} s$ equals $1$ determine a $\mu_{\ell^n}$-torsor, whence a canonical element of $\rH^1(X_{\kbar}, \Z/\ell^n(1))$.

Thus there is a canonical, functorial map \begin{equation}\label{root_to_H1}  \phi_{\Abel_* \widetilde{x}_t - \Abel_* \widetilde{x}_s} [\ell^n] \to \rH^1(X_{\kbar}, \Z/\ell^n(1)) .\end{equation} 

Composing \eqref{path_to_root} and \eqref{root_to_H1} determines a $\Gal(\kbar/k)$-equivariant map $$ \pi_1^{\ell} (\Jac^{-1} \widetilde{X}_{\kbar}, \Abel_* \bound(\kbar)) \to \rH^1(X_{\kbar}, R(1)),$$ whence a map  $$\theta: \mathcal{F}_R \pi_1^{\ell} (\Jac^{-1} \widetilde{X}_{\kbar}, \Abel_* \bound(\kbar)) \to \rH^1(X_{\kbar}, R(1)),$$ which we will show to be an isomorphism.

Applying $\rH^*(-, R(1))$ to the pair $(X_{\kbar}, \widetilde{X}_{\kbar})$ gives the exact sequence \begin{align*}\label{HRpairXtildeXMES} \ldots \to \rH^1(X_{\kbar}, \widetilde{X}_{\kbar}; R(1)) \to \rH^1(\widetilde{X}_{\kbar}, R (1)) \to  \rH^1(X_{\kbar}, R(1))\\ \to \rH^2(X_{\kbar}, \widetilde{X}_{\kbar}; R(1)) \to \rH^2(\widetilde{X}_{\kbar}, R(1)) \to \rH^2(X_{\kbar}, R(1)) \to \ldots .\end{align*}

By purity \cite[VI, Theorem~5.1]{Milnebook}, the relative cohomology groups are computed by $ \rH^1(\widetilde{X}_{\kbar}, X_{\kbar}; R(1)) \cong 0$, and $\rH^2(\widetilde{X}_{\kbar}, X_{\kbar} ; R(1)) \cong \rH^0( \bound_{\kbar}, R)$. \hidden{For this, $\rH^*(\widetilde{X}_{\kbar}, X_{\kbar}; F) = \tilde{\rH}^*(\operatorname{Thom}(\bound, \nu(\bound,\widetilde{X}_{\kbar})), F )$. In the presence of a Thom class, the spectra $\operatorname{Thom}(\bound, \nu(\bound,\widetilde{X}_{\kbar})) \wedge F $ trivializes the twisting of $\operatorname{Thom}(\bound, \nu(\bound,\widetilde{X}_{\kbar}))$ giving $\operatorname{Thom}(\bound, \nu(\bound,\widetilde{X}_{\kbar})) \wedge F  \cong \operatorname{Thom}(\bound, \bound \times \mathbb{A}^c ) \wedge F $. We assert that $\operatorname{Thom}(\bound, \bound \times \mathbb{A}^c ) \cong {\mathbb{P}^1}^{\wedge c} \wedge \bound$. Since $\mathbb{P}^1 = \Sigma \G_m$ and $\tilde{\rH}^*(\G_m )= \Zhat (-1)$ for $*= 1$ and $0$ otherwise, we have that $\tilde{\rH}^* ({\mathbb{P}^1}^{\wedge c} \wedge \bound, F) = \tilde{\rH}^{* - 2c} (\bound, F(-c))$. In total, this gives that $\rH^*(X_{\kbar}, \widetilde{X}_{\kbar}; F) =  \widetilde{\rH}^{*-2c}(\bound,F(-c))$.} By Poincar\'e duality \cite[Theorem~11.1]{Milnebook}, there is a unique isomorphism $\rH^2(\widetilde{X}_{\kbar}, R(1)) \to R$ taking the cycle class of a point to $1$  \cite[Cycle 2.1.5]{sga4andhalf}. Since $X_{\kbar}$ is affine of dimension $1$, the group $ \rH^2(X_{\kbar}, R(1)) $ vanishes. Substituting these computations into the above gives \begin{equation}\label{XXtildeES} 0 \to  \rH^1(\widetilde{X}_{\kbar}, R (1)) \to \rH^1(X_{\kbar}, R(1)) \to \bigoplus_{\bound(\kbar)} R \stackrel{\oplus \operatorname{id}}{\longrightarrow} R \to 0. \end{equation}

It is straight-forward to check that $\theta$ induces a map of exact sequences from \eqref{groupoidSES} with $\pi = \pi_1^{\ell} (\Jac^{-1} \widetilde{X}_{\kbar}, \Abel_* \bound(\kbar))$ to \eqref{XXtildeES} which is the identity on $\bigoplus_{\bound(\kbar)} R$  and $R$. By \eqref{path_to_root}, there is a map $$\pi_1 \to \varprojlim_n \Jac^0 \widetilde{X} (\kbar) [\ell^n],$$ which is an isomorphism because $\Jac^0 \widetilde{X}_{\kbar}$ is an abelian variety \cite[Chapter~IV, 18, Serre--Lang Theorem]{Mumford_AV}. Since the N\'eron--Severi group of $\widetilde{X}_{\kbar}$ is torsion-free, $\Jac^0 \widetilde{X} (\kbar) [\ell^n] \cong \Jac \widetilde{X}(\kbar) [\ell^n].$ By the moduli definition of the Picard functor, its $\ell^n$-torsion points over $\kbar$ are $ \Jac \widetilde{X}(\kbar) [\ell^n] \cong \rH^1(\widetilde{X}_{\kbar}, \G_m)[\ell^n]$. The Kummer exact sequence $$1 \to \mu_{\ell^n} \to \G_m \stackrel{b \mapsto b^{\ell^n}}\to \G_m \to 1 $$ and the fact that $\kbar^*$ is $\ell^n$-divisible shows $$ \rH^1(\widetilde{X}_{\kbar}, \G_m)[\ell^n] \cong \rH^1(\widetilde{X}_{\kbar}, \mu_{\ell^n}),$$ whence $$\pi_1 \cong  \rH^1(\widetilde{X}_{\kbar}, \Z_{\ell}(1)).$$ Thus $\theta$ is an isomorphism.

\end{proof}

\section{Homology of $\Jacbar \Xsing $}\label{section_H1Jbar} 

This section gives a canonical isomorphism of $\Gal(\kbar/k)$-modules $$\rH_1(\Jacbar^{-1} \Xsing_{\kbar},R) \to \rH^1(X_{\kbar}, R(1)).$$ 

By Proposition~\ref{Proposition: UnivesalHomeo}, there is a universal homeomorphism $\overline{P}^{\natural} \to \Jacbar^{-1} \Xsing$, where $\overline{P}^{\natural}$ is given by the pushout square $$\xymatrix{ \Jac^{0} \widetilde{X} \times f^{-1}(x_0) \ar[d] \ar[rr] && \Pres \Xsing \ar[d]^{{\PtoJbar}^{\natural}} \\ \Jac^{0} \widetilde{X} \ar[rr]\hidden{^{\mathcal{L} \mapsto f_* (\mathcal{L} \otimes \mathcal{I}_{f^{-1}(x_0)})}} && \overline{P}^{\natural}  ,}$$ which after base change to $\kbar$ becomes the pushout square \begin{equation}\label{Jnatural_kbar_CD} \xymatrix{ \coprod_{\bound(\kbar)} \Jac^{0} \widetilde{X}_{\kbar} \ar[d]_{\coprod \operatorname{id}} \ar[r]^{\hat{\epsilon}_{\kbar}} & \Pres \Xsing_{\kbar} \ar[d] \\ \Jac^{0} \widetilde{X}_{\kbar} \ar[r] & \overline{P}^{\natural}_{\kbar} .}\end{equation} 


Assume $\widetilde{X}_{\kbar}$ is not genus $0$. Let $e$ be a geometric point of $\Jac^0 \widetilde{X}_{\kbar}$ whose image is the identity element. The maps $\Jac^0 \widetilde{X}_{\kbar} \to \Pres \Xsing_{\kbar}$ comprising $\hat{\epsilon}_{\kbar}$ send $e$ to a set of geometric points of $\Pres \Xsing_{\kbar}$, denoted $\mathcal{E}$. Since all points of $\mathcal{E}$ have the same image in $\overline{P}^{\natural}_{\kbar}$, there is an induced $\Gal(\kbar/k)$-equivariant map $$ \pi_1^{\ell}( \Pres \Xsing_{\kbar},  \mathcal{E}) \to \pi_1^{\ell} (\overline{P}^{\natural}_{\kbar}).$$ Composing with the Hurewicz map  gives $\pi_1^{\ell}( \Pres \Xsing_{\kbar},  \mathcal{E}) \to \rH_1(\overline{P}^{\natural}_{\kbar}, R),$ whence \begin{equation}\label{FrpiPtoH1J}\mathcal{F}_R \pi_1^{\ell}( \Pres \Xsing_{\kbar},  \mathcal{E}) \to \rH_1(\overline{P}^{\natural}_{\kbar}, R).\end{equation}

By Proposition~\ref{Prop: PresExists}, the presentation scheme is a projective bundle $${\PtoPic}_{\kbar}: \Pres \Xsing_{\kbar} \to \Jac^{-1} \widetilde{X}_{\kbar}$$ such that the composite map $$\coprod_{\bound(\kbar)} \Jac^{0} \widetilde{X}_{\kbar} \to \Pres \Xsing_{\kbar} \to \Jac^{-1} \widetilde{X}_{\kbar}$$ is the coproduct over $\widetilde{x}$ in $\bound(\kbar)$ of the maps $L \mapsto L \otimes I_{\widetilde{x}}.$ This projective bundle induces a $\Gal(\kbar/k)$-equivariant isomorphism $\pi_1^{\ell}( \Pres \Xsing_{\kbar},  \mathcal{E}) \to \pi_1^{\ell}(\Jac^{-1} \widetilde{X}_{\kbar}, \Abel_* \bound(\kbar))$ by the homotopy exact sequence for the fundamental group \cite[X, Corollary~1.4]{sga1}, which extends to this isomorphism of groupoids, for instance by \eqref{groupoidSES}.  Here we are using that the geometric points of  $\Abel_* \bound (\kbar)$ are distinct, which follows from the hypothesis that $\widetilde{X}_{\kbar}$ is not genus $0$. Applying Theorem~\ref{Prop_groupoid_to_H1} defines $$\sigma: \rH^1(X_{\kbar}, R(1)) \to  \rH_1(\overline{P}^{\natural}_{\kbar}, R).$$

\begin{pr}	\label{Prop_H1XR(1)toH1jbar_iso}
	$\sigma$ is an isomorphism.
\end{pr}

\begin{proof}
The Mayer--Vietoris sequence (see Appendix~\ref{appendix_MV}) corresponding to \eqref{Jnatural_kbar_CD}  gives the exact sequence of $\Gal(\kbar/k)$-modules \begin{align*} &\bigoplus_{\bound(\kbar)} \rH_1( \Jac^{0} \widetilde{X}_{\kbar}, R) \to \rH_1( \Jac^{0} \widetilde{X}_{\kbar}, R )  \bigoplus \rH_1( \Pres \Xsing_{\kbar},R) \to \rH_1 (\overline{P}^{\natural}_{\kbar}, R) \\ \to  & \bigoplus_{\bound(\kbar)} \rH_0( \Jac^{0} \widetilde{X}_{\kbar}, R ) \to  \rH_0( \Jac^{0} \widetilde{X}_{\kbar}, R ) \bigoplus \rH_0( \Pres \Xsing_{\kbar}, R) \to  \rH_0 (\overline{P}^{\natural}_{\kbar}, R) \end{align*} by Theorem~\ref{homology_MV_appendix}.

Let $* = 0$ or $1$. Applying $\rH_*$ to $\hat{\epsilon}_{\kbar}$ is the $\bound(\kbar)$-fold coproduct of a fixed isomorphism, as can be seen by noting that $\rH_*( \Pres \Xsing_{\kbar}, R) \to \rH_*(\Jac^{-1} \widetilde{X}_{\kbar}, R)$ is an isomorphism since $\Pres \Xsing_{\kbar} \to \Jac^{-1} \widetilde{X}_{\kbar}$ is a projective bundle, and $\rH_* (L \mapsto L \otimes I_{\widetilde{x}} : \Jac^{0} \widetilde{X}_{\kbar} \to \Jac^{-1} \widetilde{X}_{\kbar})$ is an isomorphism. 

This gives the exact sequence $$0 \to \rH_1( \Pres \Xsing_{\kbar},R) \to  \rH_1 (\overline{P}^{\natural}_{\kbar}, R) \to   \bigoplus_{\bound(\kbar)} \rH_0( \Jac^0 \widetilde{X}_{\kbar}, R ) \to  \rH_0( \Jac^0 \widetilde{X}_{\kbar}, R ) \to 0 $$ where the map $$\bigoplus_{\bound(\kbar)} \rH_0( \Jac^0 \widetilde{X}_{\kbar}, R) \to  \rH_0( \Jac^0 \widetilde{X}_{\kbar}, R )$$ is the $\bound(\kbar)$-fold coproduct of the identity map.
 
Identifying $\rH_0( \Jac^0 \widetilde{X}_{\kbar}, R )$ with $R$, we obtain \begin{equation}\label{H1JbarMES} 0 \to \rH_1( \Pres \Xsing_{\kbar},R ) \to  \rH_1 (\overline{P}^{\natural}_{\kbar}, R) \to   \bigoplus_{\bound(\kbar)} R \to  R \to 0 .\end{equation}

By \eqref{H1JbarMES}, the definition of the Mayer--Vietoris sequence, and \eqref{groupoidSES}, it follows that \eqref{FrpiPtoH1J} is an isomorphism, proving the proposition. 
\end{proof}


If $\widetilde{X}_{\kbar}$ is genus $0$, we have that $\Pic^0 \widetilde{X}_{\kbar} \cong \Spec \kbar$. By the Mayer--Vietoris sequence (Theorem~\ref{homology_MV_appendix}) corresponding to \eqref{Jnatural_kbar_CD}, we have an exact sequence $$0 \to \rH_1(\Pres \Xsing_{\kbar}, R) \to \rH_1(\Jacbar^{-1} \Xsing_{\kbar},R) \to \oplus_{\bound(\kbar)} R \to R \to 0.$$ By Proposition~\ref{Prop: PresExists},  the presentation scheme is a projective bundle over $\Spec \kbar$, whence $\rH_1(\Pres \Xsing_{\kbar}, R) =0$, giving the desired isomorphism in this case by Remark~\ref{H1for_genus0_tildeX}.

\section{The Abel map gives Poincar\'e duality}\label{sectionAJPD}

In this section we prove the main theorem that the Altman-Kleiman Abel map realizes Poincar\'e duality.

Let $\Tr : \rH^2(\Xsing_{\kbar}, R(1)) \cong \rH^2_c(X_{\kbar}, R(1)) \to R$ denote the trace map, sending the class of a point to $1$ \cite[Cycle 2.1.5]{sga4andhalf}.

Let $\wp:  \rH_1(\Xsing_{\kbar}, R) \to \rH^1(X_{\kbar}, R(1))$ be the Poincar\'e duality isomorphism characterized by $\Tr (\gamma \cup \wp(\lambda) ) = \langle \lambda, \gamma \rangle$ for all $\lambda$ in $ \rH_1(\Xsing_{\kbar}, R)$ and $\gamma$ in $\rH^1(\Xsing_{\kbar}, R) \cong \rH^1_c(X_{\kbar}, R)$, where $ \langle -, - \rangle$ denotes the tautological pairing between $\rH_1$ and $\rH^1$, defined: if $l$ in $\pi_1^{\ell}(\Xsing_{\kbar})$ represents $\lambda$, then $l$ acts by addition by $ \langle \lambda, \gamma \rangle$ on the fiber of the torsor classified by $\gamma$.

\begin{tm}	\label{minus_wp_is_sigma_inverse_H1_Abel}
	$- \wp = \sigma^{-1} \rH_1(\Abel_{\kbar}).$
\end{tm}

To prove Theorem~\ref{minus_wp_is_sigma_inverse_H1_Abel}, we equip ourselves with three lemmas.

For a product $Y \times Z$, let $\operatorname{pr}_1: Y \times Z \to Y$ and $\operatorname{pr}_2: Y \times Z \to Z$ denote the projections. Let $N$ be a positive integer not divisible by the characteristic of $k$.

Let $g: X_{\kbar} \to \widetilde{X}_{\kbar}$ denote the open immersion, resulting in another open immersion $$g \times 1: X_{\kbar} \times X_{\kbar} \to \widetilde{X}_{\kbar} \times X_{\kbar}.$$  

The diagonal $\Delta$ of $X_{\kbar} \times X_{\kbar}$ defines a class $\cl(\Delta)$ in $\rH^2_{\Delta}(X_{\kbar} \times X_{\kbar}, \mu_N)$ by \cite[VI, Section~6, Theorem~6.1]{Milnebook}. Furthermore, $\Delta$ is closed in $\widetilde{X}_{\kbar} \times X_{\kbar}$, allowing us to apply excision \cite[III, Section~1, Proposition~1.27]{Milnebook} which defines an isomorphism $\rH^2_{\Delta}(X_{\kbar} \times X_{\kbar}, \mu_N) \cong \rH^2_{\Delta}(\widetilde{X}_{\kbar} \times X_{\kbar}, (g \times 1)_! \mu_N)$\hidden{the isomorphism $\rH^2_{\Delta}(X_{\kbar} \times X_{\kbar}, \mu_N) \cong \rH^2_{\Delta}(\widetilde{X}_{\kbar} \times X_{\kbar}, (g \times 1)_! \mu_N)$ follows from excision: the \'etale map $X_{\kbar} \times X_{\kbar} \to \widetilde{X}_{\kbar} \times X_{\kbar}$ has the property that $\Delta$ has closed image and the restriction of $g \times 1j$ to $\Delta$ is an isomorphism. Furthermore, $g \times 1(X_{\kbar} \times X_{\kbar} - \Delta) \subset \widetilde{X}_{\kbar} \times X_{\kbar} - \Delta$. Therefore by \cite[III Section~1, Proposition~1.27]{Milnebook}, $\rH^*_{\Delta}(\widetilde{X}_{\kbar} \times X_{\kbar}, (g \times 1)_! \mu_N) \cong \rH^* (X_{\kbar} \times X_{\kbar}, (g \times 1)^* (g \times 1)_! \mu_N) \cong \rH^* (X_{\kbar} \times X_{\kbar}, \mu_N)$}.  

The adjunction $(g \times 1)_! ,(g \times 1)^*$ and the map $\mu_N \to \operatorname{pr}_1^* g^* g_! \mu_N \cong (g \times 1)^* \operatorname{pr}_1^* g_! \mu_N$ on $X_{\kbar} \times X_{\kbar}$\hidden{induced by pulling back $\mu_N \to g^* g_! \mu_N$ on $X_{\kbar}$  by $\operatorname{pr}_1$} define a map $ (g \times 1)_! \mu_N \to \operatorname{pr}_1^* (g_! \mu_N)$ which is an isomorphism because it induces isomorphisms on all stalks. The isomorphisms $$(g \times 1)_! \mu_N \cong \operatorname{pr}_1^* (g_! \mu_N) \cong \operatorname{pr}_1^*(g_! \Z/N \otimes \mu_N) \cong \operatorname{pr}_1^*(g_! \Z/N ) \otimes \mu_N \cong \operatorname{pr}_1^*(g_! \Z/N ) \otimes \operatorname{pr}_2^* \mu_N$$ allow us to apply the K\"unneth formula to $\rH^*( \widetilde{X}_{\kbar} \times X_{\kbar}, (g \times 1)_! \mu_N)$, from which we obtain $$\rH^*(\widetilde{X}_{\kbar} \times X_{\kbar}, (g \times 1)_! \mu_N) \cong \rH^*(\widetilde{X}_{\kbar}, g_! \Z/n) \otimes \rH^*(X_{\kbar}, \mu_N) \cong \rH^*_c( X_{\kbar}, \Z/N) \otimes \rH^*(X_{\kbar}, \mu_N).$$  This allows us to speak of $(i,j)$ K\"unneth components of elements of $\rH^*(\widetilde{X}_{\kbar} \times X_{\kbar}, (g \times 1)_! \mu_N)$. 

Let $c^{1,1}$ be the $(1,1)$ K\"unneth component of the image of $\cl(\Delta)$ under $$\rH^2_{\Delta}(\widetilde{X}_{\kbar} \times X_{\kbar}, (g \times 1)_! \mu_N) \to \rH^2(\widetilde{X}_{\kbar} \times X_{\kbar}, (g \times 1)_! \mu_N).$$ We may view $c^{1,1}$ as an element of $\rH^1_c (X_{\kbar}, \rH^1(X_{\kbar}, \mu_N)) \cong \rH^1(\widetilde{X}_{\kbar}, g_! \Z/N) \otimes \rH^1(X_{\kbar}, \mu_N)$.

The diagonal $\widetilde{\Delta}$ of $\widetilde{X}_{\kbar} \times \widetilde{X}_{\kbar}$ determines a class $\cl(\widetilde{\Delta})$ in $\rH^2_{\widetilde{\Delta}}(\widetilde{X}_{\kbar} \times \widetilde{X}_{\kbar}, \mu_N)$. Let $d^{1,1}$ denote the $(1,1)$ K\"unneth component of the image of $\cl(\widetilde{\Delta})$ under $$\rH^2_{\widetilde{\Delta}}(\widetilde{X}_{\kbar} \times \widetilde{X}_{\kbar}, \operatorname{pr}_1^* \Z/N \otimes \operatorname{pr}_2^* \mu_N) \to \rH^2(\widetilde{X}_{\kbar} \times \widetilde{X}_{\kbar},\operatorname{pr}_1^* \Z/N \otimes \operatorname{pr}_2^* \mu_N),$$ which we view as an element of $\rH^1 (\widetilde{X}_{\kbar}, \rH^1(\widetilde{X}_{\kbar}, \mu_N))$.

Since there is a map of smooth pairs $(\Delta, X_{\kbar} \times X_{\kbar}) \to (\widetilde{\Delta}, \tilde{X}_{\kbar} \times \tilde{X}_{\kbar})$, we have \begin{equation}\label{clDelta_Deltatilde_pullback}(g \times g)^* \cl (\widetilde{\Delta}) = \cl (\Delta) \end{equation} by \cite[VI, Section~6, Theorem 6.1(c)]{Milnebook}. 

Our first lemma rewrites \eqref{clDelta_Deltatilde_pullback} in terms of $f: \widetilde{X} \to X^+$. The viewpoint is that $\rH^1_c (X_{\kbar}, \rH^1(X_{\kbar}, \mu_N))$ and $\rH^1 (\widetilde{X}_{\kbar}, \rH^1(\widetilde{X}_{\kbar}, \mu_N))$ classify certain torsors on $X^+_{\kbar}$ and $\widetilde{X}_{\kbar}$ respectively, and \eqref{clDelta_Deltatilde_pullback} computes the pullback of the torsor classified by $c^{1,1}$ under $f$ in terms of $d^{1,1}$. 

To describe pullback by $f: \widetilde{X} \to X^+$ more specifically, note that $f$ induces $$f^*: \rH^1(X^+_{\kbar}, \rH^1(X_{\kbar}, \mu_N)) \cong\rH^1_c(X_{\kbar}, \rH^1(X_{\kbar}, \mu_N)) \to \rH^1(\widetilde{X}_{\kbar}, \rH^1(X_{\kbar}, \mu_N)).$$ Equivalently, tensoring the map $g_! \Z/N \to \Z/N$ of sheaves on $\widetilde{X}$ with $\rH^1(X_{\kbar}, \mu_N)$ and applying $\rH^1(\widetilde{X}_{\kbar}, -)$ gives $f^*$. 

We introduce one last piece of notation. The open immersion $g: X \to \tilde{X}$ induces a map $$g^* :  \rH^1(\widetilde{X}_{\kbar}, \mu_N) \to \rH^1(X_{\kbar}, \mu_N),$$ and applying $\rH^1(\widetilde{X}_{\kbar},-)$ gives a map $$\rH^1(\widetilde{X}_{\kbar}, g^*): \rH^1(\widetilde{X}_{\kbar}, \rH^1(\widetilde{X}_{\kbar}, \mu_N)) \to  \rH^1(\widetilde{X}_{\kbar}, \rH^1(X_{\kbar}, \mu_N)).$$ 

\begin{lm}\label{gd=fc}
$$\rH^1(\widetilde{X}_{\kbar}, g^*) d^{1,1} = f^* c^{1,1} .$$
\end{lm}

\begin{proof}
We may assume $k = \kbar$. The following diagram is commutative: the top two rows commute by commutativity of boundary maps associated to the map of pairs $$(\Delta, X \times X) \to (\widetilde{\Delta}, \tilde{X} \times \tilde{X});$$ the second and third rows by naturality of the K\"unneth formula; the right trapezoid also by naturality of the K\"unneth formula; and the bottom triangle obviously: $$\xymatrix{ \rH^2_{\widetilde{\Delta}}(\widetilde{X} \times \widetilde{X}, \mu_N) \ar[r] \ar[d] & \ar[d]  \rH^2_{\Delta}(X \times X, \mu_N) \cong  \rH^2_{\Delta}(\widetilde{X} \times X, g_! \Z/N \boxtimes \mu_N) \ar[r] & \rH^2(\widetilde{X} \times X, g_! \Z/N \boxtimes \mu_N) \ar[dl] \ar[dd]_{\textrm{K\"unneth}} \\ \rH^2(\widetilde{X} \times \widetilde{X}, \mu_N) \ar[r] \ar[d]^{\textrm{K\"unneth}}& \rH^2(X \times X, \mu_N) \ar[d]^{\textrm{K\"unneth}}\\ \rH^1(\widetilde{X}, \Z/N) \otimes \rH^1(\widetilde{X}, \mu_N) \ar[r]^{g^* \otimes g^*} \ar[rrd]^{\operatorname{id} \otimes g^*}& \rH^1(X, \Z/N) \otimes \rH^1(X, \mu_N)& \rH^1(\widetilde{X}, g_! \Z/n) \otimes \rH^1(X, \mu_N) \ar[d]^{f^* \otimes \operatorname{id}} \\ & & \ar[ul] \rH^1(\widetilde{X}, \Z/N) \otimes  \rH^1(X, \mu_N)},$$ where the notation $(-) \boxtimes (-) $ is an abbreviation for $\operatorname{pr}_1^*(-) \otimes \operatorname{pr}_2^*(-)$. 

The image of $\operatorname{cl}(\widetilde{\Delta}) \in \rH^2_{\widetilde{\Delta}}(\widetilde{X} \times \widetilde{X}, \mu_N)$ under the top two horizontal morphisms followed by the right vertical morphism is $f^* c^{1,1}$. The image of $\operatorname{cl}(\widetilde{\Delta})$ under the left two vertical morphisms followed by the bottom diagonal morphism is $\rH^1(\widetilde{X}, g^*) d^{1,1},$ showing the proposition.
\end{proof}

The next lemma is a close analogue of \cite[Dualit\'e, Proposition~3.4]{sga4andhalf}, whose proof is almost identical, but we include it for completeness.  

The trace map $\Tr: \rH^2_c(X_{\kbar}, \mu_N) \to \Z/N$ defines a map $$\Tr' : \rH^1(\widetilde{X}_{\kbar}, g_! \Z/N) \otimes \rH^1(X_{\kbar}, \mu_N) \otimes \rH^1(X_{\kbar}, \mu_N) \to \rH^1(X_{\kbar}, \mu_N)$$ \begin{equation}\label{trace'def} a \otimes b \otimes c \mapsto \Tr(a \cup c) b.\end{equation}

\begin{lm}\label{analogueProp3.4}
For all $x \in \rH^1(X_{\kbar}, \mu_N)$, we have $\Tr' (-c^{1,1} \otimes x) = x$.
\end{lm}

\begin{proof}
We may assume $k = \kbar$. The above isomorphism $(g \times 1)_! \mu_N \cong \operatorname{pr}_1^* g_! \mu_N$ induces an isomorphism $(g \times 1)_! \mu_N \otimes \mu_N \cong g_! \mu_N \boxtimes \mu_N$ of sheaves on $\widetilde{X} \times X$. We obtain a K\"unneth formula for $\rH^*(\tilde{X} \times X, (g \times 1)_! \mu_N \otimes \mu_N) \cong \rH^*(\widetilde{X}, g_! \mu_N) \otimes \rH^*(X, \mu_N)$, allowing us to define $\Tr_3$ $$\xymatrix{\Tr_3: \rH^3(\widetilde{X} \times X, (g \times 1)_! \mu_N \otimes \mu_N) \cong \rH^3(\widetilde{X} \times X,  g_! \mu_N \boxtimes \mu_N)\ar[d] \\ \rH^2(\widetilde{X}, g_!  \mu_N ) \otimes \rH^1(X, \mu_N) \ar[d] \\  \rH^1(X, \mu_N)}$$ which takes the $(2,1)$ K\"unneth component and then applies trace. 

There is a cup product \cite[Cycle~1.2.4]{sga4andhalf}$$\rH^2(\widetilde{X} \times X, (g \times 1)_! \mu_N) \times \rH^1(X \times X, \mu_N) \to \rH^3(\widetilde{X} \times X, (g \times 1)_! \mu_N \otimes \mu_N).$$ Note that \begin{equation}\label{Tr'toTr3eqn}\Tr' (-c^{1,1} \otimes x) =  -\Tr_3(-c^{1,1} \cup \operatorname{pr}_1^* x) = \Tr_3(c^{1,1} \cup  \operatorname{pr}_1^* x),\end{equation} where the minus sign after the first equality comes from permuting cochains of degree $1$ in \eqref{trace'def}.

Since $c^{1,1} \cup \operatorname{pr}_1^* x$ and $\operatorname{cl}(\Delta) \cup \operatorname{pr}_1^* x$ have the same $(2,1)$ K\"unneth component, \begin{equation}\label{c11toclDelta}  \Tr_3(c^{1,1} \cup  \operatorname{pr}_1^* x) =  \Tr_3(\operatorname{cl}(\Delta) \cup  \operatorname{pr}_1^* x).\end{equation}

The cup product on the right hand side may be reinterpreted as the cup product $$ \rH^2_{\Delta}( \widetilde{X} \times X,(g \times 1)_! \mu_N) \times \rH^1 (\Delta, \mu_N) \to \rH^3_{\Delta}(\widetilde{X} \times X,  (g \times 1)_! \mu_N \otimes \mu_N )$$ defined in \cite[Cycle~1.2.2.2]{sga4andhalf}.

Since $\operatorname{pr}_1 = \operatorname{pr}_2$ when restricted to $\Delta$, we therefore have \begin{equation}\label{swap_p1p2} \Tr_3(\operatorname{cl}(\Delta) \cup  \operatorname{pr}_1^* x) =  \Tr_3(\operatorname{cl}(\Delta) \cup  \operatorname{pr}_2^* x) = \Tr_4 (\operatorname{cl}(\Delta)) \cup x = x,\end{equation} where $$\Tr_4: \rH^2(\widetilde{X} \times X, (g \times 1)_! \mu_N) \cong \rH^2(\widetilde{X} \times X, g_! \mu_N \boxtimes \Z/N) \to \Z/N $$ takes the $(2,0)$ K\"unneth component and applies $\Tr$.

Combining \eqref{Tr'toTr3eqn}, \eqref{c11toclDelta}, and \eqref{swap_p1p2} completes the proof.
\end{proof}

Recall the notation $\mathcal{F}_R$ for the free $R$-module on a groupoid from Section~\ref{Section:CohXfromPic}. The next lemma identifies $\rH_1(\Xsing_{\kbar}, R)$ in terms of the free $R$-module associated to paths in $\widetilde{X}_{\kbar}$ between geometric points of $\partial X$.

\begin{lm}\label{H1XsingisFRpi1Xtildebound}
The map $\mathcal{F}_R  \pi_1^{\ell}(\widetilde{X}_{\kbar}, \partial X(\kbar)) \to \rH_1(\Xsing_{\kbar}, R)$ induced by $$ \pi_1^{\ell}(\widetilde{X}_{\kbar}, \partial X(\kbar)) \stackrel{f_*}{\to} \pi_1^{\ell}(\Xsing_{\kbar}, \overline{x}_0) \to \rH_1(\Xsing_{\kbar}, R)$$ is an isomorphism. 
\end{lm}

\begin{proof}
The Mayer--Vietoris sequence (Theorem~\ref{homology_MV_appendix} of Appendix~\ref{appendix_MV}) corresponding to $$\begin{CD}
					\bound_{\kbar}	 				@>>> 				\widetilde{X}_{\kbar} \\
					@VVV										@VfVV \\
					\Spec \kbar 					@>\overline{x}_{0}>>	\Xsing_{\kbar}
				\end{CD}$$ gives the exact sequence of $\Gal(\kbar/k)$-modules \begin{align*}\oplus_{\partial X(\kbar)} \rH_1(\Spec \kbar, R) \to \rH_1( \Spec \kbar, R) \oplus \rH_1(\widetilde{X}_{\kbar}, R) \to \rH_1(X^+_{\kbar}, R) \\ \to  \oplus_{\partial X(\kbar)} \rH_0(\Spec \kbar, R) \to \rH_0( \Spec \kbar, R) \oplus  \rH_0(\widetilde{X}_{\kbar}, R) \to \rH_0(X^+_{\kbar}, R), \end{align*} which we can rewrite as $$0 \to \rH_1(\widetilde{X}_{\kbar}, R) \to \rH_1(X^+_{\kbar}, R) \to \oplus_{\partial X(\kbar)} R \to R \to 0$$ using the isomorphisms $\rH_1(\Spec \kbar, R) = 0$, $\rH_0(\Spec \kbar, R) \cong \rH_0(\widetilde{X}_{\kbar}, R) \cong R$, and noting that the map $$\oplus_{\partial X(\kbar)} \rH_0(\Spec \kbar, R) \to \rH_0(\Spec \kbar, R) \oplus \rH_0(\widetilde{X}_{\kbar}, R)$$ is identified with the coproduct over $\partial X(\kbar)$ of the diagonal map $R \to R \oplus R$.

Associated to $\mathcal{F}_R \pi_1^{\ell}(\widetilde{X}_{\kbar}, \partial X)$ is the exact sequence \eqref{groupoidSES} $$0 \to \mathcal{F}_R \pi_1^{\ell}(\widetilde{X}_{\kbar}, \overline{x}_0) \to \mathcal{F}_R \pi_1^{\ell}(\widetilde{X}_{\kbar}, \partial X) \to \oplus_{\partial X_{\kbar}} R \to R \to 0,$$ which is compatible with the Mayer--Vietoris sequence in the sense that the diagram $$\xymatrix{0 \ar[r] & \rH_1(\widetilde{X}_{\kbar}, R) \ar[r] & \rH_1(X^+_{\kbar}, R) \ar[r] & \oplus_{\partial X(\kbar)} R \ar[r] & R \ar[r] & 0 \\ 0 \ar[r] & \ar[u] \mathcal{F}_R \pi_1^{\ell}(\widetilde{X}_{\kbar}, \overline{x}_0) \ar[r] & \mathcal{F}_R \pi_1^{\ell}(\widetilde{X}_{\kbar}, \partial X) \ar[r] \ar[u] & \oplus_{\partial X(\kbar)} R \ar[u]^{\operatorname{id}} \ar[r] & R \ar[u]^{\operatorname{id}} \ar[r] & 0} $$ commutes.

Since $ \mathcal{F}_R \pi_1^{\ell}(\widetilde{X}_{\kbar}, \overline{x}_0) \to \rH_1(\widetilde{X}_{\kbar}, R)$ is an isomorphism, $\mathcal{F}_R \pi_1^{\ell}(\widetilde{X}_{\kbar}, \partial X) \to \rH_1(X^+_{\kbar}, R)$ is also an isomorphism.
\end{proof}

We now give the proof of Theorem~\ref{minus_wp_is_sigma_inverse_H1_Abel}.

\begin{proof}(of Theorem~\ref{minus_wp_is_sigma_inverse_H1_Abel}.)
Suppose $R = \Z/ \ell^n$, and let $N = \ell^n$. By definition of the cohomology of an $\ell$-adic sheaf \cite[p.~163--164]{Milnebook}, it suffices to prove the theorem in this case.

Since $\wp$, $\sigma$, and $\rH_1(\Abel_{\kbar})$ are $R$-module morphisms, by Lemma~\ref{H1XsingisFRpi1Xtildebound} it suffices to show that $$- \wp (f_* \gamma)= \sigma^{-1} \rH_1(\Abel_{\kbar})(f_*\gamma)$$ for $\gamma \in \pi_1^{\ell}(\widetilde{X}_{\kbar}, \bound)$.

By Lemma~\ref{analogueProp3.4}, $$- \wp (f_* \gamma) = \Tr'(c^{1,1} \otimes \wp (f_* \gamma)).$$ 

By the definition of $\wp$, $$\Tr'(c^{1,1} \otimes \wp (f_* \gamma)) = \langle f_* \gamma , c^{1,1}\rangle.$$ 

$c^{1,1}$ classifies a $\rH^1(X_{\kbar}, R(1))$-torsor, which we will denote $Y \to \Xsing_{\kbar}$. As in Section~\ref{Section:CohXfromPic}, let $\phi_{x}(T)$ denote the fiber of a torsor $T$ at a geometric point $x$. Pulling back $Y$ by $f$ gives a $\rH^1(X_{\kbar}, R(1))$-torsor $f^* Y$ with isomorphisms $\phi_{\tilde{x}} f^* Y \cong \phi_{\overline{x}_0} Y$ for all $\tilde{x} \in \bound (\kbar)$, allowing us to speak of the monodromy of $f^*Y$ along $\gamma$ as an element $\langle \gamma, f^*c^{1,1} \rangle$ in $\rH^1(X_{\kbar}, R(1))$. Furthermore, $$ \langle f_* \gamma , c^{1,1}\rangle =  \langle \gamma, f^*c^{1,1} \rangle.$$

By Lemma~\ref{gd=fc}, $f^* Y$ is classified by $\rH^1(\widetilde{X}_{\kbar}, g^*) d^{1,1}$ in $\rH^1(\widetilde{X}_{\kbar}, \rH^1(X_{\kbar}, R(1)))$, which allows us to write $$\langle \gamma, f^* c^{1,1}\rangle = \langle \gamma, \rH^1(\widetilde{X}_{\kbar}, g^*) d^{1,1}\rangle, $$ giving \begin{equation}\label{pf*gamma_as_monodromy} - \wp (f_* \gamma) = \langle \gamma, \rH^1(\widetilde{X}_{\kbar}, g^*) d^{1,1}\rangle \end{equation} by combining with the three previous equations. 

Now consider $\sigma^{-1} \rH_1(\Abel_{\kbar})(f_*\gamma)$. By the construction of the lifted Abel map (Definition~\ref{df_lifted_Abel_map}\hidden{see paragraph immediately following the definition}), $$\PtoJbar \circ \AbelLift = \Abel \circ f,$$ whence $\rH_1(\Abel_{\kbar})(f_*\gamma)$ is the element of homology represented by $${\PtoJbar}_* \circ \AbelLift_* \gamma \in \pi_1^{\ell} (\Jacbar^{-1} \Xsing,\Abel_* \overline{x}_0).$$

Consider an element of the homology group $\rH_1(\Jacbar^{-1} \Xsing_{\kbar}, R)$ represented by an element of $\pi_1^{\ell}(\Jacbar^{-1} \Xsing_{\kbar},\Abel_* \overline{x}_0)$ which is the image under $\PtoJbar$ of a path in $\Pres \Xsing_{\kbar}$. The image under $\sigma^{-1}$ of such an element  is particularly easy to evaluate. Here is the resulting description of $\sigma^{-1}{\PtoJbar}_* \circ \AbelLift_* \gamma$. As in the proof of Theorem~\ref{Prop_groupoid_to_H1}, let $a$ be a geometric point of $\Jac^1 \widetilde{X}_{\kbar}$, defining $+_a: \Jac^{-1} \widetilde{X}_{\kbar} \to \Jac^0 \widetilde{X}_{\kbar}$. Let $m_{N}: \Jac^0 \widetilde{X}_{\kbar} \to \Jac^0 \widetilde{X}_{\kbar}$ denote the finite \'etale cover given by multiplication by $N$. Let $\widetilde{x}_s$ and $\widetilde{x}_t$ in $\bound(\kbar)$ be the source and target of $\gamma$. Since $\gamma$ is a path, $\gamma$ induces a morphism between fibers $$\phi_{\widetilde{x}_s}(\Abel^* +_a^* m_N) \to \phi_{\widetilde{x}_t}(\Abel^* +_a^* m_N),$$ which is given by addition by an element $\langle \gamma , \Abel^* +_a^* m_N \rangle$ of $\phi_{\Abel_* \widetilde{x}_t - \Abel_* \widetilde{x}_s} [m_N]$.  By construction, $\sigma^{-1} {\PtoJbar}_* ( \AbelLift_* \gamma)$ is the image under \eqref{root_to_H1} of $\langle \gamma , \Abel^* +_a^* m_N \rangle$.

By \cite[Dualit\'e, Proposition~3.2]{sga4andhalf}, $d^{1,1}$ classifies the pullback of $m_N$ by $+_a \circ \Abel$. (The notation in loc.~cit.~is that $u$ is the pullback of $m_N$ by the negative of $+_a \circ \Abel$, whence the appearance of the sign.) Thus the torsor classified by $ \rH^1(\widetilde{X}_{\kbar}, g^*) d^{1,1}$ is $$(\Abel^* +_a^* m_N) \times_{\rH^1(\widetilde{X}_{\kbar}, R(1))} \rH^1(X_{\kbar}, R(1)),$$ which identifies the action of $\gamma$ on the fibers of $\rH^1(\widetilde{X}_{\kbar}, g^*) d^{1,1}$ with its action on the fibers of $d^{1,1}$. In particular, the map induced by $\gamma$ from the fiber of $\rH^1(\widetilde{X}_{\kbar}, g^*) d^{1,1}$ over $\widetilde{x}_s$ to the fiber over $\widetilde{x}_t$ is addition by $\langle \gamma , \Abel^* +_a^* m_N \rangle$. The isomorphisms $\phi_{\widetilde{x}_s} f^* Y \cong \phi_{\widetilde{x}_t} f^* Y$ are compatible with \eqref{root_to_H1}, giving that the image of  $\langle \gamma , \Abel^* +_a^* m_N \rangle$ under \eqref{root_to_H1} is $\langle\gamma, \rH^1(\widetilde{X}_{\kbar}, g^*) d^{1,1}\rangle$. Combining with \eqref{pf*gamma_as_monodromy} completes the proof.
\end{proof}

\appendix \section{Cohomology in terms of the fundamental group of $\Pic$}\label{appendixa}

In this appendix we prove Proposition~\ref{natl_iso_pi_1Pic=H1}, identifying $ \rH^1((-)_{\kbar}, \Z_{\ell}(1))$ with the $\ell$-\'etale fundamental group of the Picard scheme, as well as Proposition~\ref{Pic0_to_H^1_NS[ell^N]}, giving a similar description of $\rH^1((-)_{\kbar}, \Z/N(1))$.

\noindent{\bf Proposition~\ref{natl_iso_pi_1Pic=H1}.} {\em Let $k$ be a perfect field. There is a natural isomorphism of functors from proper, geometrically connected schemes over $k$ to $\Gal(\kbar/k)$-modules $$\pi_1^{\ell} (\Pic^0 (-)_{\kbar}, e) \cong \rH^1((-)_{\kbar}, \Z_{\ell}(1)),$$ for $\ell$ a prime not equal to the characteristic of $k$.}

\begin{rmk}
In the statement of Proposition~\ref{natl_iso_pi_1Pic=H1}, the $\ell$-\'etale fundamental group $\pi_1^{\ell}$ can be replaced by $\rH_1(-,\Z_{\ell})$.
\end{rmk}

Summarily, this proposition is proven by applying the Kummer exact sequence to obtain $\rH^1((-)_{\kbar}, \Z/\ell^n(1)) \cong \Pic (-)_{\kbar}[\ell^n],$ and then relating torsion points and $\pi_1^{\ell}$ for algebraic groups.

Let $k$ and $\ell$ be as above.  Let $p$ denote the characteristic of $k$, which could be $0$. $\rH^*$ denotes \'etale cohomology. By an {\em algebraic group}, we mean a connected, smooth $k$-group scheme. An algebraic group is automatically geometrically connected by \cite[$VI_A$, Proposition~2.4]{SGA3_8_11}.

\begin{lm} \label{Gkbardiv}
	Let $G$ be a commutative algebraic group over $k$, and $N$ an integer which is prime-to-$p$. Then the $\kbar$-points of $G$ form an $N$-divisible group.
\end{lm}

\begin{proof}
The subcategory of groups whose $\kbar$-points are $N$-divisible is closed under extensions. The $\kbar$-points of $\G_m$, $\G_{a}$, and all abelian varieties are $N$-divisible. Thus the lemma follows by the classification of (connected) commutative, algebraic groups \cite[XVII, Theorem~7.2.1]{SGA3_8_11}. 
\end{proof}

Let $\pi_1^{(p')}$ denote the prime-to-$p$ \'etale fundamental group, and $\pi_1^{\ell}$ denote its maximal pro-$\ell$ quotient as above. By \cite[Proposition~1.1 and Remark~4.3]{Br_Sz}, it follows that: 

\begin{pr} \label{pi1comgrp}
	 If $G$ is a commutative algebraic group over $k$, then there is a natural isomorphism $\pi_{1}^{(p')}(G_{\kbar},e) \cong \varprojlim G[N](\bar{k})$, where $G[N]$ denotes the $N$-torsion and $N$ runs over positive integers which are prime to $p$.
\end{pr}

\begin{proof}
 By \cite[Proposition~1.1]{Br_Sz}, $\pi_{1}^{(p')}(G_{\kbar},e)$ is abelian, whence isomorphic to the maximal prime-to-$p$ quotient of $\varprojlim_{Y} \Gal(Y/G_{\kbar})$, where $Y$ runs over the finite, \'etale, abelian Galois covers of $G_{\kbar}$. Thus $$\pi_{1}^{(p')}(G_{\kbar},e) \cong  \varprojlim_Y \Gal(Y/G_{\kbar})^{(p')}$$ where $\Gal(Y/G_{\kbar})^{(p')}$ denotes the maximal prime-to-$p$ quotient of $\Gal(Y/G_{\kbar})$. The quotient $\Gal(Y/G_{\kbar})^{(p')}$ is naturally isomorphic to the Galois group of a finite, \'etale, abelian cover $Z_Y \to G_{\kbar}$ of degree prime-to-$p$. It follows that $\pi_{1}^{(p')}(G_{\kbar},e) \cong \varprojlim_Z \Gal(Z/ G_{\kbar})$ where $Z$ runs over\hidden{ the finite, \'etale, abelian Galois covers of $G_{\kbar}$ of exponent prime-to-$p$} such covers. By \cite[Remark~4.3]{Br_Sz}, for $N$ prime-to-$p$, $$0 \to G_{\kbar}[N] \to G_{\kbar} \to G_{\kbar} \to 0$$ is the largest abelian, \'etale, Galois cover of $G$ of exponent $N$, whence $$\pi_{1}^{(p')}(G_{\kbar},e) \cong \varprojlim G[N](\bar{k})$$ as claimed.
\end{proof}

Let $Z$ be a proper scheme over $k$. By \cite[Section~8.2, Theorem 3]{BLR}\hidden{this seems to only assume proper}, $\Pic Z$ is represented by a locally finite type scheme over $k$, which is automatically a commutative group object. Let $\Pic^0 Z$ denote the connected component of the identity $e$ of $\Pic Z$, and $(\Pic^0 Z)_{\red}$ denote the reduced closed subscheme. By \cite[Proposition 4.6.1]{egaIV_2}, $(\Pic^0 Z)_{\red} \times_k (\Pic^0 Z)_{\red}$ is reduced, and it follows that $(\Pic^0)_{\red}$ is a commutative group scheme. By generic smoothness, $(\Pic^0)_{\red}$ is an algebraic group.

For $Z$ as above, let $\NS Z_{\kbar}$ denote the N\'eron--Severi group of connected components of $\Pic Z_{\kbar}$ \cite[p.~234]{BLR}, $$\NS Z_{\kbar}= \Pic Z_{\kbar} (\kbar) / \Pic^0 Z_{\kbar} (\kbar) \cong \Pic Z (\kbar)/ \Pic^0 Z (\kbar).$$ 

\begin{pr} \label{Pic0_to_H^1_NS[ell^N]}
	For $N$ prime-to-$p$, there is a short exact sequence of functors $$0 \to \Pic ^0 (-)_{\kbar}[N](\kbar) \to \rH^1((-)_{\kbar}, \Z/N(1)) \to \NS (-)_{\kbar} [N] \to 0 $$ from proper, geometrically connected schemes over $k$ of locally finite type to $\Z/N$-modules with $\Gal(\kbar/k)$-action.
\end{pr}

\begin{proof}
The Kummer exact sequence $$1 \to \mu_N \to \G_m \to \G_m \to 1 $$ and the fact that $\kbar^*$ is $N$-divisible gives an identification $$  \rH^1((-)_{\kbar}, \Z /N (1)) = \rH^1((-)_{\kbar}, \mu_N)\cong \rH^1((-)_{\kbar}, \G_m)[N].$$ 

By the definition of $\Pic$, its $N$-torsion points over $\kbar$ are $$\Pic (-)_{\kbar} [N] (\kbar) = \rH^1((-)_{\kbar}, \G_m)[N].$$ 

Thus the claim is equivalent to showing a natural exact sequence $$0 \to \Pic ^0 (-)_{\kbar}[N](\kbar) \to\Pic (-)_{\kbar} [N] (\kbar) \to \NS (-)_{\kbar} [N] \to 0 .$$

By the definition of the N\'eron--Severi group $$0 \to \Pic^0 (-)_{\kbar}(\kbar) \to  \Pic (-)_{\kbar}(\kbar) \to \NS (-)_{\kbar} \to 0$$ is exact, and it follows that $$0 \to \Pic^0 (-)_{\kbar}[N](\kbar) \to  \Pic (-)_{\kbar}[N](\kbar) \to \NS (-)_{\kbar}[N] \to \Pic^0 (-)_{\kbar}(\kbar) / N\Pic^0 (-)_{\kbar}(\kbar)$$ is also exact.

Note the canonical isomorphism $\Pic^0 (-)_{\kbar}(\kbar) \cong (\Pic^0 (-)_{\kbar})_{\red}(\kbar)$. By the above, $(\Pic^0 (-)_{\kbar})_{\red}$ is an algebraic group, so by Lemma~\ref{Gkbardiv}, $$(\Pic^0 (-)_{\kbar})_{\red}(\kbar)/N(\Pic^0 (-)_{\kbar})_{\red}(\kbar) \cong 0,$$ proving the claim.
\end{proof}

\begin{proof} (of Proposition~\ref{natl_iso_pi_1Pic=H1}) 
By Proposition~\ref{pi1comgrp}, we have a natural isomorphism of functors $ \pi_1^{\ell}(-,e) \circ \Pic ^0  (-)_{\kbar}= \varprojlim_N \Pic^0(-)_{\kbar} [N](\kbar)$ where $N$ ranges over the powers of $\ell$, because the $\kbar$-points and \'etale fundamental group of $\Pic^0(-)_{\kbar}$ can be identified with those of its reduction, and $(\Pic ^0  (-)_{\kbar})_{\red}$ is a commutative algebraic group. 

By Proposition~\ref{Pic0_to_H^1_NS[ell^N]}, we have exact sequences  $$0 \to \Pic^0(-)_{\kbar} [N]\to \rH^1((-)_{\kbar}, \Z/N(1)) \to \NS (-)_{\kbar} [N] \to 0 $$ where $N$ ranges over the powers of $\ell$.

\hidden{By Lemma~\ref{Gkbardiv} and the identification of the $\kbar$-points of $\Pic^0$ and $\Pic^0_{\red}$, the multiplication by $M$ maps $  \Pic^0 [N  M] \to \Pic^0 [N]$ for $M$ a positive integer are surjective, whence $\varprojlim_N^1  \Pic^0 [N] = 0$. Thus we have an exact sequence \begin{equation*}0 \to \varprojlim_N \Pic^0 [N] \to \varprojlim_N \rH^1(-, \Z/N(1)) \to \varprojlim_N \NS[N] \to 0 \end{equation*} where $N$ runs over the powers of $\ell$.

This yields $$0 \to \pi_1^{\ell} (\Pic^0, e) \to \rH^1(X_{\kbar}, \Z_{\ell}(1)) \to \varprojlim_N \NS [N]  \to 0 $$ where $N$ runs over the powers of $\ell$ and the transition maps are given by multiplication by powers of $\ell$\hidden{By \cite[p.~164]{Milnebook} $ \rH^1(-, \Z_{\ell}) =\varprojlim \rH^1(-, \Z/\ell^N) $ -- indeed this is the definition.}}

Taking the inverse limit gives the exact sequence $$\xymatrix{0 \ar[r] & \pi_1^{\ell} (\Pic^0(-)_{\kbar}, e) \cong \varprojlim_N \Pic^0(-)_{\kbar} [N] \ar[d] \\ & \rH^1((-)_{\kbar}, \Z_{\ell}(1)):=\varprojlim_N \rH^1((-)_{\kbar}, \Z/N(1)) \ar[r] & \varprojlim_N \NS (-)_{\kbar} [N].}$$

Since $\NS (-)_{\kbar}$ is finitely generated \cite[8.4, Theorem~7]{BLR}, \hidden{the needed hypothesis are : $Z \to \Spec k$ proper, and locally of finite presentation} multiplication by $\ell^n$ is the $0$ map for large enough $n$, and if follows that $\varprojlim_N \NS (-)_{\kbar} [N] = 0$.

\end{proof}

\section{A Mayer--Vietoris sequence for a pushout of schemes}\label{appendix_MV}

Consider diagrams of finite type schemes over $k$ of the form $$\xymatrix{ V \ar[r]^i \ar[d]_p & W\\ Z &} $$ with $i$ a closed immersion, and $p$ finite. By \cite[Theorem~5.4]{ferrand03}, the pushout \begin{equation}\label{VWZWbar_pushout_CD}\xymatrix{ V \ar[r]^i \ar[d]_p & W \ar[d]^{\overline{p}} \\ Z \ar[r]_{\overline{i}} & \overline{W}}  \end{equation} exists and commutes with base change by field extensions \cite[Lemma~4.4]{ferrand03}. The resulting morphisms $\overline{p}$ and $\overline{i}$ are finite, with $\overline{i}$ a closed immersion \cite[Theorem~5.4]{ferrand03}. 

This appendix proves a Mayer--Vietoris sequence in cohomology for pushouts of this form, and then obtains a truncation of such a sequence in homology. The truncation results from the fact that we only define $\rH_1$ and $\rH_0$. It is not mathematically essential.

\begin{tm} \label{MV_cohomology_appendix}
	There is a functorial association of a long exact sequence of $\Gal(\kbar/k)$-modules $$ \ldots \to \rH^n(\overline{W}_{\kbar},R)  \to \rH^n(W_{\kbar},R) \oplus \rH^n(Z_{\kbar},R) \to \rH^n (V_{\kbar}, R) \to \rH^{n+1}(\overline{W}_{\kbar},R) \to \ldots$$ to a pushout diagram \eqref{VWZWbar_pushout_CD}.
\end{tm}

Here the adjective ``functorial" means that a map of pushout diagrams of the form \eqref{VWZWbar_pushout_CD} induces a map of long exact sequences.

\begin{proof}
Let $f: V \to \overline{W}$ be defined $f =  \overline{i} p= \overline{p} i$. Let $R(V_{\kbar})$ denote the constant sheaf with stalk $R$ on the \'etale site of $V_{\kbar}$ for $R=\Z/\ell^m$ and the corresponding $\ell$-adic sheaf for $R=\Z_{\ell}$ \cite[p.~163]{Milnebook}, with similar definitions for $W$, $Z$, and $\overline{W}$. Note that $\overline{p}^* R(\overline{W}_{\kbar}) \cong R(W_{\kbar})$, giving a natural map $R(\overline{W}_{\kbar}) \to \overline{p}_* R(W_{\kbar})$.

The horizontal arrows in the commutative diagram $$\xymatrix{ f_* R(V_{\kbar}) &\ar[l] \overline{p}_* R(W_{\kbar})\\ \ar[u] \overline{i}_*R(Z_{\kbar})& \ar[l] \ar[u] R(\overline{W}_{\kbar})  } $$ are surjective. Indeed, to check surjectivity, it is enough to verify the condition on stalks \cite[II, Theorem~2.15(c)]{Milnebook}, and surjectivity on stalks follows from \cite[II, Corollary~3.5(a)]{Milnebook}. The resulting morphism of short exact sequences \begin{equation}\label{MV_app_SES_sheaves} \xymatrix{ 0 & \ar[l] f_* R(V_{\kbar}) &\ar[l] \overline{p}_* R(W_{\kbar}) & \ar[l] \Ker \overline{p}_*(  i_* i^*\leftarrow \operatorname{id}) & \ar[l] 0 \\ 0 & \ar[l] \ar[u] \overline{i}_*R(Z_{\kbar})& \ar[l] \ar[u] R(\overline{W}_{\kbar}) & \ar[l] \ar[u]^N \Ker( \overline{i}_* \overline{i}^* \leftarrow \operatorname{id} ) & \ar[l] 0 } \end{equation} is such that $N$ is an isomorphism. To see that $N$ is an isomorphism, note that by \cite[II, Theorem~2.15(b) and II, Corollary~3.5(c)]{Milnebook} pushforward by a finite morphism is an exact functor between categories of \'etale sheaves, whence $\Ker \overline{p}_*(i_* i^* \leftarrow \operatorname{id}) = \overline{p}_* \Ker (i_* i^* \leftarrow \operatorname{id})$. Let $j: W-V \hookrightarrow W$ be the open immersion corresponding to the complement of $V$. The exact sequence \cite[II, p.~76]{Milnebook} $$0 \to j_! j^* R \to R \to i_* i^* R \to 0 $$ gives $ \Ker (i_* i^* \leftarrow \operatorname{id}) \cong j_! R((W-V)_{\kbar})$. Similarly, $\Ker ( \overline{i}_* \overline{i}^* \leftarrow \operatorname{id} ) \cong \overline{j}_! R((\overline{W} - Z)_{\kbar})$, where $\overline{j}: \overline{W} -Z \hookrightarrow \overline{W}$ is the open immersion. Since $\overline{p}$ induces an isomorphism $W - V \to \overline{W} - Z$ (\cite[Theorem~5.4(d)]{ferrand03}) and $\overline{p}^{-1} Z \subset V$ (\cite[Theorem~5.4(b)]{ferrand03}), $N$ is an isomorphism by checking on all stalks.

A morphism of short exact sequences of sheaves with isomorphic kernels $$\xymatrix{0 & \ar[l] A &\ar[l] B &\ar[l] C & \ar[l] 0 \\ 0 &\ar[l]  A' \ar[u] &\ar[l]  B' \ar[u]& \ar[l] C' \ar[u]^{\cong} &\ar[l]  0  } $$ gives rise to a Mayer--Vietoris sequence $$\ldots \to \rH^n(B') \to \rH^n(B) \oplus \rH^n(A') \to \rH^n(A) \to \rH^{n+1}(B') \to \ldots .$$ Applying this principle to \eqref{MV_app_SES_sheaves} gives a long exact sequence \begin{equation}\label{messier_MV_appendix}\xymatrix{ \ldots \ar[r] & \rH^n(\overline{W}_{\kbar},R)  \ar[r] & \rH^n(\overline{W}_{\kbar},\overline{p}_*R(W_{\kbar})) \oplus \rH^n(\overline{W}_{\kbar}, \overline{i}_*R(Z_{\kbar})) \ar[r] &  \rH^n (\overline{W}_{\kbar}, f_* R(V_{\kbar}))  \\ \ar[r] & \rH^{n+1}(\overline{W}_{\kbar},R) \ar[r] & \ldots.}\end{equation}
 
Since pushforward by a finite morphism is an exact functor between categories of \'etale sheaves, $\rH^n(W_{\kbar}, R)  \cong \rH^n(\overline{W}_{\kbar}, \overline{p}_* R(W_{\kbar})) $, and there are similar expressions for $\rH^n (V_{\kbar},R) $ and $ \rH^n(Z_{\kbar},R)$. Making these substitutions into \eqref{messier_MV_appendix} completes the proof.

\end{proof}

Let $Y$ be a $k$-scheme. Let $\rH_0(Y_{\kbar}, R)$ denote the free $R$-module on the connected components $\pi_0(Y_{\kbar})$ of $Y_{\kbar}$. Since $\pi_0(Y_{\kbar})$ has a continuous Galois action, $\rH_0(Y_{\kbar}, R)$ is a $\Gal(\kbar/k)$-module. For $y \in \pi_0(Y_{\kbar})$, let $(Y_{\kbar})_y$ denote the corresponding connected component. For $R = \Z/\ell^n$ (or $R$ a finite abelian group), define $\rH_1(Y_{\kbar}, R)$ as the direct sum $$\rH_1(Y_{\kbar}, R) \cong \oplus_{y \in \pi_0Y_{\kbar}} ( \pi_1(Y_{\kbar})_y^{\ab} \otimes R),$$ where $\pi_1(Y_{\kbar})_y^{\ab}$ denotes the abelianization of the \'etale fundamental group of $ \pi_1(Y_{\kbar})_y$. For $R = \Z_{\ell}$, define $\rH_1(Y_{\kbar}, \Z_{\ell})$ as $$\rH_1(Y_{\kbar}, \Z_{\ell})\cong \oplus_{y \in \pi_0Y_{\kbar}}~\pi_1(Y_{\kbar})_y^{\ell-\ab},$$ where $ \pi_1(Y_{\kbar})_y^{\ell-\ab}$ denotes the abelianization of the $\ell$-\'etale fundamental group, i.e.,~the abelianization of the maximal pro-$\ell$ quotient of the \'etale fundamental group. Note that $\pi_1(Y_{\kbar})_y^{\ab}$ is independent of the choice of base point. The $\Gal(\kbar/k)$-action on \'etale paths in $Y_{\kbar}$ gives $\rH_1(Y_{\kbar}, R)$ the structure of a $\Gal(\kbar/k)$-module. More specifically, for each $y \in \pi_0(Y_{\kbar})$ choose a base point $b_y$, which may not or may not have a $k$-rational point as its image, with which to define $\pi_1((Y_{\kbar})_y, b_y)$. Then $\sigma \in \Gal(\kbar/k)$ determines a map $\pi_1((Y_{\kbar})_y, b_y) \to \pi_1(Y_{\kbar}, \sigma_* b_y)$. Then $\sigma_* b_y$ has image in some connected component of $Y_{\kbar}$, say the connected component $w$. The canonical isomorphism $$\pi_1((Y_{\kbar})_w, b_w)^{\ab} \cong \pi_1((Y_{\kbar})_w, \sigma_* b)^{\ab}$$ allows $\sigma$ to act on $\rH_1(Y_{\kbar},R)$.

\hidden{

\begin{claim}\label{Hom_to_Zellm_exact}
	 Let $m_1, m_2, m$ be non-negative integers with $m_1 + m_2 = m$. Then applying $\Hom_{\Z/m} (-, \Z/m)$ to the short exact sequence $$0 \to \Z/\ell^{m_1} \to \Z/\ell^m \to \Z/ \ell^{m_2} \to 0 $$ is the short exact sequence $$0 \leftarrow \Z/\ell^{m_1} \leftarrow \Z/\ell^m \leftarrow \Z/ \ell^{m_2} \leftarrow 0 .$$
\end{claim}

\begin{proof}
There is a unique subgroup $H$ of $\Z/\ell^m$ of order $\Z/ \ell^{m_1}$. The inclusion induces an isomorphism $\Hom(\Z/\ell^{m_1}, H) \cong \Hom(\Z/\ell^{m_1}, \Z/\ell^m).$ Since $\Hom(\Z/\ell^{m_1}, H) \cong \Hom(\Z/\ell^{m_1}, \Z/\ell^{m_1}) \cong \Z/\ell^{m_1}$, we have $\Hom(\Z/\ell^{m_1}, \Z/\ell^m) \cong \Z/\ell^{m_1}$. Apply  $\Hom_{\Z/m} (-, \Z/m)$ to the injection $\Z/\ell^{m_1} \hookrightarrow \Z/\ell^m$. The generator $\sigma$ of $\Hom_{\Z/\ell^m} (\Z/\ell^m, \Z/\ell^m)$ is determined by $\sigma(1) = 1$. The image of $\sigma$ sends $1$ to $\sigma(\ell^{m - m_2})$, i.e.,~$1 \mapsto \ell^{m-m_2}$. Note that this is the generator of $\Hom_{\Z/\ell^m} (\Z/\ell^{m_1}, \Z/\ell^m)$. The claim follows. \end{proof}

\end{proof}}

We prove the following technical lemma to avoid considering the question of finite generation for $\pi_1^{\ell}$. Note that $\rH_1(Y_{\kbar}, \Z/\ell^m)$ has a natural topology induced from the topology on \'etale fundamental groupoids.\hidden{ To see this, it suffices to put a topology on $\prod_{y \in \pi_0(Y_{\kbar})} (\pi_1((Y_{\kbar})_y)^{\ab} \otimes \Z/\ell^m)$, whence to put a topology on $\pi_1((Y_{\kbar})_y)^{\ab} \otimes \Z/\ell^m$. This tensor product has an induced topology, which is described as the quotient topology from $\pi_1((Y_{\kbar})_y)^{\ell, \ab} \otimes \Z/\ell^m \to \pi_1({Y}_y)^{\ab} \otimes \Z/\ell^m$ or, equivalently, as the inverse limit topology on $\pi_1((Y_{\kbar})_y)^{\ab} \otimes \Z/\ell^m \cong \varprojlim (\Gal(\overline{Y}/(Y_{\kbar})_y \otimes \Z/\ell^m)$.} For a topological $\Z/\ell^m$-module $M$, let $M^{\star} = \Hom(M, \Z/\ell^m )$ be the continuous homomorphisms $M \to \Z/\ell^m$. 

\begin{lm} \label{fMgamma}
For each $y \in \pi_0(Y_{\kbar})$, let $M_y \subseteq \rH_1((Y_{\kbar})_y, \Z/\ell^m)$ be a closed sub-$\Z/\ell^m$-module, and let $M = \oplus_y M_y$. For any $\gamma$ in $\rH_1(Y_{\kbar}, \Z/\ell^m)$ with $\gamma \notin M$, there exists $f \in \rH_1(Y_{\kbar}, \Z/\ell^m)^{\star}$ such that $f(\gamma) \neq 0$ and $f(m) = 0$ for all $m \in M$.
\end{lm}

\begin{proof}
We may assume $k = \kbar$. By definition, $$\rH_1(Y, \Z/\ell^m) = \oplus_{y \in \pi_0(Y)} (\pi_1({Y}_y)^{\ab} \otimes \Z/\ell^m).$$ Since $\gamma$ is not in $M$ and $M = \oplus_y M_y$, there exists $y \in \pi_0(Y)$ such that the image $\gamma_{y}$ of $\gamma$ under $$\rH_1(Y, \Z/\ell^m) \to \prod_{y \in \pi_0(Y)} (\pi_1({Y}_y)^{\ab} \otimes \Z/\ell^m) \to \pi_1({Y}_y)^{\ab} \otimes \Z/\ell^m  $$ is not contained in $M_y$. Using the map $(\pi_1({Y}_y)^{\ab} \otimes \Z/\ell^m)^{\star} \to \rH_1(Y_{\kbar}, \Z/\ell^m)^{\star}$, we may assume that $Y$ is connected. Since $\pi_1(Y)^{\ab} \otimes \Z/\ell^m$ is profinite, there is a finite quotient such that the image of $\gamma$ is not contained in the image of $M$, as otherwise the intersection over all finite quotients of those elements of $M$ with the same image as $\gamma$ would be non-empty by compactness of $M$, which in turn would imply that $\gamma$ is in $M$. Thus there exists an abelian finite \'etale cover $\overline{Y} \to Y$ such that the image of $\gamma$ under \begin{equation}\label{H_1_to_Galois_projection}\rH_1(Y, \Z/\ell^m)  \to  \Gal(\overline{Y} / Y) \otimes \Z/\ell^m\end{equation} is not contained in the image of $M$. The $\Z/\ell^m$-module $(\Gal(\overline{Y} / Y) \otimes \Z/\ell^m)/M$ is finitely generated (even finite), whence \hidden{(Claim~\ref{Hom_to_Zellm_exact})} the natural map  $$(\Gal(\overline{Y} / Y) \otimes \Z/\ell^m)/M \to (((\Gal(\overline{Y} / Y) \otimes \Z/\ell^m)/M)^{\star})^{\star} $$ is an isomorphism. In particular, we may choose $\overline{f} \in ((\Gal(\overline{Y} / Y) \otimes \Z/\ell^m)/M)^{\star}$ such that $\overline{f}(\gamma) \neq 0$. Let $f$ be the image of $\overline{f}$ under the map $$((\Gal(\overline{Y} / Y) \otimes \Z/\ell^m)/M)^{\star} \to (\Gal(\overline{Y} / Y) \otimes \Z/\ell^m)^{\star} \to\rH_1(Y_{\kbar}, \Z/\ell^m)^{\star}.$$
\end{proof}

As a corollary, we obtain: 

\begin{lm} \label{H1Zellm_star_star_injective}
	The natural map $\rH_1(Y_{\kbar}, \Z/\ell^m) \to (\rH_1(Y_{\kbar}, \Z/\ell^m)^{\star})^{\star}$ is injective.
\end{lm}
\hidden{Here is a proof independent of Lemma~\ref{fMgamma}

\begin{proof}
We may assume $k = \kbar$. By definition, $\rH_1(Y, \Z/\ell^m) = \oplus_{y \in \pi_0(Y)} (\pi_1({Y}_y)^{\ab} \otimes \Z/\ell^m)$. Suppose that $\gamma$ is a nonzero element of $\rH_1(Y, \Z/\ell^m)$. There exists $y \in \pi_0(Y)$ and an abelian finite \'etale cover $\overline{Y} \to Y_y$ such that  $\gamma$ has a nonzero image under \begin{equation}\label{H_1_to_Galois_projection}\rH_1(Y, \Z/\ell^m) \to \prod_{y \in \pi_0(Y)} (\pi_1({Y}_y)^{\ab} \otimes \Z/\ell^m) \to \pi_1({Y}_y)^{\ab} \otimes \Z/\ell^m \to  \Gal(\overline{Y} / Y_y) \otimes \Z/\ell^m.\end{equation} The $\Z/\ell^m$-module $\Gal(\overline{Y} / Y_y) \otimes \Z/\ell^m$ is finitely generated, whence \hidden{(Claim~\ref{Hom_to_Zellm_exact})} the natural map  $$\Gal(\overline{Y} / Y_y) \otimes \Z/\ell^m \to ((\Gal(\overline{Y} / Y_y) \otimes \Z/\ell^m)^{\star})^{\star} $$ is an isomorphism. In particular, we may choose $f \in (\Gal(\overline{Y} / Y_y) \otimes \Z/\ell^m)^{\star}$ such that $f(\gamma) \neq 0$. Using \eqref{H_1_to_Galois_projection}, $f$ determines an element of $\rH_1(Y_{\kbar}, \Z/\ell^m)^{\star}$ such that $f(\gamma) \neq 0$, showing the claimed injectivity.
\end{proof}}

We use these lemmas to show the Mayer--Vietoris sequence in homology used above. Assume that $V_{\kbar}$, $Z_{\kbar}$,$W_{\kbar}$ and $\overline{W}_{\kbar}$ have finitely many connected components. 

\begin{tm} \label{homology_MV_appendix}
	There is a functorial exact sequence of $\Gal(\kbar/k)$-modules \begin{align*}\rH_1 (V_{\kbar}, R)  \to \rH_1(W_{\kbar},R) \oplus \rH_1(Z_{\kbar},R) \to  \rH_1(\overline{W}_{\kbar},R) \to \\  \rH_{0} (V_{\kbar}, R)  \to \rH_0(W_{\kbar},R) \oplus \rH_0(Z_{\kbar},R) \to  \rH_0(\overline{W}_{\kbar},R) \to 0.\end{align*}
\end{tm}

\begin{proof}
Let $n=0,1$.  Functoriality of $\rH_n$ defines maps $$ \rH_n(i) \times \rH_n(p): \rH_n(V_{\kbar}, R) \to  \rH_n(W_{\kbar},R) \oplus \rH_n(Z_{\kbar},R)$$ $$ \rH_n(i) \times \rH_n(p) (x) = \rH_n(i) x \oplus \rH_n(p) x$$   $$-\rH_n(\overline{p}) \oplus \rH_n(\overline{i}): \rH_n(W_{\kbar},R) \oplus \rH_n(Z_{\kbar},R) \to \rH_n(\overline{W}_{\kbar}, R)$$ $$- \rH_n(\overline{p}) \oplus \rH_n(\overline{i})(x) = - \rH_n(\overline{p}) (x)+ \rH_n(\overline{i})(x).$$

Suppose first that $R = \Z/\ell^m$. The isomorphism $ \rH^1((Y_{\kbar})_y, R) \cong \Hom(\pi_1(Y_{\kbar})_y, R )$ between the first \'etale cohomology group and continuous homomorphisms from the fundamental group (this isomorphism can be obtained from, for instance, \cite[III, Section~4, Section~2 Corollary~2.10 and I, Section~5]{Milnebook}) determines an isomorphism \begin{equation}\label{H1YR=HomH_1R} \rH^1(Y_{\kbar}, R) \cong \Hom(\rH_1(Y_{\kbar}, R), R ),\end{equation} where $\Hom$ again denotes continuous homomorphisms.

The isomorphism \eqref{H1YR=HomH_1R} determines a map $\rH_1(\overline{W}_k, R) \to  \rH^1(\overline{W}_{\kbar}, R)^{\star}$. Dualizing the boundary map in Theorem~\ref{MV_cohomology_appendix} gives a map $$ \rH^1(\overline{W}_{\kbar}, R)^{\star} \to \rH^0(\overline{W}_{\kbar}, R)^{\star}.$$ By assumption $\overline{W}_{\kbar}$ has finitely many connected components, identifying $\prod_{\pi_0 \overline{W}_{\kbar}} R$ and $\oplus_{\pi_0 \overline{W}_{\kbar}} R$, and giving an isomorphism $$\rH^0(\overline{W}_{\kbar}, R)^{\star} \cong \rH_0 (\overline{W}_{\kbar}, R).$$ Use the resulting composite $$ \rH_1(\overline{W}_k, R) \to  \rH^1(\overline{W}_{\kbar}, R)^{\star} \to \rH^0(\overline{W}_{\kbar}, R)^{\star} \cong \rH_0 (\overline{W}_{\kbar}, R) $$ to define the boundary map in the statement we are proving. Together with $\rH_n(i) \times \rH_n(p)$ and $- \rH_n(\overline{p}) \oplus \rH_n(\overline{i})$, this defines the sequence in the statement.

Note that, with this definition, applying $\Hom( -, R)$ to the sequence of Theorem~\ref{homology_MV_appendix} gives the sequence of Theorem~\ref{MV_cohomology_appendix}. It thus follows from Lemma~\ref{H1Zellm_star_star_injective} that the composition of adjacent maps in the sequence of Theorem~\ref{homology_MV_appendix} is $0$, i.e.,~this sequence is a complex. We claim this complex is exact. Suppose to the contrary that exactness fails at $$ M_{j+1} \to M_j \to M_{j-1},$$ i.e.,~there exists $\gamma$ in $\Ker(M_j \to M_{j-1})$ which is not in $\operatorname{Image}(M_{j+1} \to M_j).$ Choose $f \in M_j^{\star}$ vanishing on the image of $M_{j+1}$ and nonzero on $\gamma$, using Lemma~\ref{fMgamma}. Then $f$ is in the kernel of $M_j^{\star} \to M_{j+1}^{\star}$. However, $f$ can not be in the image of $M_{j-1}^{\star} \to M_j^{\star}$ because any element of this image sends $\gamma$ to $0$.

The quotient map $\Z/\ell^{m+1} \to \Z/\ell^m$ determines maps $\rH_n (Y_{\kbar}, \Z/\ell^{m+1}) \to \rH_n (Y_{\kbar}, \Z/\ell^{m})$. It is tautological to check that these maps induce maps between the corresponding exact sequences in the statement. The case $R = \Z_{\ell}$ then follows from the isomorphisms $\rH_n (Y_{\kbar}, \Z_{\ell}) \cong \varprojlim_m \rH_n (Y_{\kbar}, \Z/\ell^m)$ and exactness of inverse limits on profinite groups.

\end{proof}

%
%
%
\bibliographystyle{AJPD}

\bibliography{Kass+Wickelgren+Abel}

\newcommand{\etalchar}[1]{$^{#1}$}
\providecommand{\bysame}{\leavevmode\hbox to3em{\hrulefill}\thinspace}
\providecommand{\MR}{\relax\ifhmode\unskip\space\fi MR }
\providecommand{\MRhref}[2]{%
  \href{http://www.ams.org/mathscinet-getitem?mr=#1}{#2}
}
\providecommand{\href}[2]{#2}
\begin{thebibliography}{SGA$4\frac{1}{2}$}

\bibitem[SGA$3{}_{III}$]{SGA3_8_11}
Michael Artin, Jean-{\'E}tienne Bertin, Michel Demazure, Peter Gabriel,
  Alexander Grothendieck, Michel Raynaud, and Jean-Pierre Serre,
  \emph{Sch\'emas en groupes. {F}asc. 3: {E}xpos\'es 8 \`a 11}, S\'eminaire de
  G\'eom\'etrie Alg\'ebrique de l'Institut des Hautes \'Etudes Scientifiques,
  vol. 1963/64, Institut des Hautes \'Etudes Scientifiques, Paris, 1964.
  \MR{MR0207705 (34 \#7520)}

\bibitem[AK80]{altman80}
Allen~B. Altman and Steven~L. Kleiman, \emph{Compactifying the {P}icard
  scheme}, Adv. in Math. \textbf{35} (1980), no.~1, 50--112. \MR{555258
  (81f:14025a)}

\bibitem[AK90]{altman90}
\bysame, \emph{The presentation functor and the compactified {J}acobian}, The
  {G}rothendieck {F}estschrift, {V}ol.\ {I}, Progr. Math., vol.~86,
  Birkh\"auser Boston, Boston, MA, 1990, pp.~15--32. \MR{1086881 (92e:14023)}

\bibitem[Ari11]{Arinkin_line_bundles}
Dima Arinkin, \emph{Cohomology of line bundles on compactified {J}acobians},
  Math. Res. Lett. \textbf{18} (2011), no.~6, 1215--1226. \MR{2915476}

\bibitem[Ari13]{Arinkin}
\bysame, \emph{Autoduality of compactified {J}acobians for curves with plane
  singularities}, J. Algebraic Geom. \textbf{22} (2013), no.~2, 363--388.
  \MR{3019453}

\bibitem[Ati61]{Atiyah}
Michael~F. Atiyah, \emph{Thom complexes}, Proc. London Math. Soc. (3)
  \textbf{11} (1961), 291--310. \MR{0131880 (24 \#A1727)}

\bibitem[Bho92]{bhosle}
Usha Bhosle, \emph{Generalized parabolic bundles and applications to
  torsionfree sheaves on nodal curves}, Ark. \textbf{30} (1992), 187--215.

\bibitem[BLR90]{BLR}
Siegfried Bosch, Werner L{\"u}tkebohmert, and Michel Raynaud, \emph{N\'eron
  models}, Ergebnisse der Mathematik und ihrer Grenzgebiete (3) [Results in
  Mathematics and Related Areas (3)], vol.~21, Springer-Verlag, Berlin, 1990.
  \MR{1045822 (91i:14034)}

\bibitem[BS13]{Br_Sz}
Michel Brion and Tam{\'a}s Szamuely, \emph{Prime-to-{$p$} \'etale covers of
  algebraic groups and homogeneous spaces}, Bull. Lond. Math. Soc. \textbf{45}
  (2013), no.~3, 602--612. \MR{3065029}

\bibitem[SGA$4\frac{1}{2}$]{sga4andhalf}
Pierre Deligne, \emph{Cohomologie \'etale}, Lecture Notes in Mathematics, Vol.
  569, Springer-Verlag, Berlin, 1977, S{\'e}minaire de G{\'e}om{\'e}trie
  Alg{\'e}brique du Bois-Marie SGA 4${1{\o}er 2}$, Avec la collaboration de
  J.~F.~Boutot, A.~Grothendieck, L.~Illusie et J.~L.~Verdier. \MR{0463174 (57
  \#3132)}

\bibitem[EGK00]{esteves00}
Eduardo Esteves, Mathieu Gagn{\'e}, and Steven~L. Kleiman, \emph{Abel maps and
  presentation schemes}, Comm. Algebra \textbf{28} (2000), no.~12, 5961--5992,
  Special issue in honor of Robin Hartshorne.

\bibitem[EGK02]{esteves2002}
\bysame, \emph{Autoduality of the compactified {J}acobian}, J. London Math.
  Soc. (2) \textbf{65} (2002), no.~3, 591--610. \MR{1895735 (2003d:14038)}

\bibitem[ER13]{Esteves_Flavio}
Eduardo Esteves and Fl{\'a}vio Rocha, \emph{Autoduality for treelike curves
  with planar singularities}, Bull. Braz. Math. Soc. (N.S.) \textbf{44} (2013),
  no.~3, 413--420. \MR{3124743}

\bibitem[Fer03]{ferrand03}
Daniel Ferrand, \emph{Conducteur, descente et pincement}, Bull. Soc. Math.
  France \textbf{131} (2003), no.~4, 553--585. \MR{2044495 (2005a:13016)}

\bibitem[Fri82]{Friedlander}
Eric~M. Friedlander, \emph{\'{E}tale homotopy of simplicial schemes}, Annals of
  Mathematics Studies, vol. 104, Princeton University Press, Princeton, N.J.,
  1982. \MR{676809 (84h:55012)}

\bibitem[EGAIV${}_2$]{egaIV_2}
Alexander Grothendieck and Jean~A. Dieudonn\'e, \emph{\'{E}l\'ements de
  g\'eom\'etrie alg\'ebrique. {IV}. \'{E}tude locale des sch\'emas et des
  morphismes de sch\'emas. {II}}, Inst. Hautes \'Etudes Sci. Publ. Math.
  (1965), no.~24, 231. \MR{MR0199181 (33 \#7330)}

\bibitem[Hu05]{Hu}
Po~Hu, \emph{On the {P}icard group of the stable {$\Bbb A^1$}-homotopy
  category}, Topology \textbf{44} (2005), no.~3, 609--640. \MR{2122218
  (2005m:14035)}

\bibitem[Isa04]{Isaksen}
Daniel~C. Isaksen, \emph{Etale realization on the {$\Bbb A^1$}-homotopy theory
  of schemes}, Adv. Math. \textbf{184} (2004), no.~1, 37--63. \MR{2047848
  (2005d:14029)}

\bibitem[Kas12]{kass12}
Jesse~L. Kass, \emph{An explicit non-smoothable component of the compactified
  {J}acobian}, J. Algebra \textbf{370} (2012), 326--343. \MR{2966842}

\bibitem[Kas13]{kass13}
\bysame, \emph{Singular curves and their compactified {J}acobians}, A
  celebration of algebraic geometry, Clay Math. Proc., vol.~18, Amer. Math.
  Soc., Providence, RI, 2013, pp.~391--427. \MR{3114949}

\bibitem[KK81]{Kleppe81}
Hans Kleppe and Steven~L. Kleiman, \emph{Reducibility of the compactified
  {J}acobian}, Compositio Math. \textbf{43} (1981), no.~2, 277--280. \MR{622452
  (83f:14027)}

\bibitem[Mil80]{Milnebook}
James~S. Milne, \emph{\'{E}tale cohomology}, Princeton Mathematical Series,
  vol.~33, Princeton University Press, Princeton, N.J., 1980. \MR{559531
  (81j:14002)}

\bibitem[Moc10]{Mochizuki}
Shinichi Mochizuki, \emph{Topics in absolute anabelian geometry i:
  generalities}, Available at
  \url{http://www.kurims.kyoto-u.ac.jp/~motizuki/papers-english.html}, 2010.

\bibitem[MRV12]{MRV}
Margarida Melo, Antonio Rapagnetta, and Filippo Viviani,
  \emph{{F}ourier-{M}ukai and autoduality for compactified {J}acobians. i},
  arXiv:1207.7233, 2012.

\bibitem[Mum70]{Mumford_AV}
David Mumford, \emph{Abelian varieties}, Tata Institute of Fundamental Research
  Studies in Mathematics, No. 5, Published for the Tata Institute of
  Fundamental Research, Bombay, 1970. \MR{0282985 (44 \#219)}

\bibitem[Mum75]{mumford75}
\bysame, \emph{Curves and their {J}acobians}, The University of Michigan Press,
  Ann Arbor, Mich., 1975. \MR{0419430 (54 \#7451)}

\bibitem[OS79]{oda}
Tadao Oda and C.~S. Seshadri, \emph{Compactifications of the generalized
  {J}acobian variety}, Trans. Amer. Math. Soc. \textbf{253} (1979), 1--90.
  \MR{536936 (82e:14054)}

\bibitem[SGAI]{sga1}
\emph{Rev\^etements \'etales et groupe fondamental ({SGA} 1)}, Documents
  Math\'ematiques (Paris) [Mathematical Documents (Paris)], 3, Soci\'et\'e
  Math\'ematique de France, Paris, 2003, S{\'e}minaire de g{\'e}om{\'e}trie
  alg{\'e}brique du Bois Marie 1960--61. [Algebraic Geometry Seminar of Bois
  Marie 1960-61], Directed by A.~Grothendieck, With two papers by M.~Raynaud,
  Updated and annotated reprint of the 1971 original [Lecture Notes in Math.,
  224, Springer, Berlin; MR0354651 (50 \#7129)]. \MR{MR2017446 (2004g:14017)}

\end{thebibliography}



\end{document}